\documentclass[12pt,a4paper,reqno]{amsart}

\usepackage{amssymb,amsmath,graphicx}
\usepackage{fancyvrb}
\usepackage{booktabs}
\usepackage{pgf}
\usepackage{tikz}
\usetikzlibrary{fit}
\newcommand\addvmargin[1]{%
  \node[fit=(current bounding box),inner ysep=#1,inner xsep=0]{};
}
\usepackage{longtable}
\usepackage{enumerate}
\usepackage{comment}
\usetikzlibrary{arrows,automata}
\usetikzlibrary{decorations.pathreplacing}

\newtheorem{theorem}{Theorem}[section]
\newtheorem{lemma}[theorem]{Lemma}

\newtheorem{proposition}[theorem]{Proposition}

\theoremstyle{definition}
\newtheorem{definition}[theorem]{Definition}
\newtheorem{example}[theorem]{Example}
\newtheorem{examples}[theorem]{Examples}
\newtheorem{remark}[theorem]{Remark}

\newtheorem{notation}[theorem]{Notation}

\newcommand{\Jonq}{Jonqui\`eres}
\newcommand{\BB}{\mathcal{B}}
\newcommand{\PP}{\mathbb{P}}
\newcommand{\K}{\mathbb{K}}

\newcommand{\infnear}{\succ}
\newcommand{\prox}{\dasharrow}

\newcommand{\satel}{\odot}
\newcommand{\notsatel}{\hbox{$\mspace{5mu}\not\mspace{-5mu}\odot\mspace{6mu}$}}

\DeclareMathOperator{\Aut}{Aut}

\DeclareMathOperator{\Bir}{Bir}
\DeclareMathOperator{\PGL}{PGL}
\DeclareMathOperator{\oq}{oql}
\DeclareMathOperator{\q}{ql}
\DeclareMathOperator{\lgth}{lgth}

\DeclareMathOperator{\mult}{mult}
\DeclareMathOperator{\h}{h}

\numberwithin{table}{section}
\numberwithin{figure}{section}

\title[On plane Cremona maps of degree 4 and length 2]{On plane Cremona maps of degree 4 and length 2}
\author{Alberto Calabri$^{*}$}
\address{$^{*}$Dipartimento di Matematica e Informatica, Università di Ferrara, Via Machiavelli 30, 44121 Ferrara, Italy}
\email{clblrt@unife.it}

\author{Nguyen Thi Ngoc Giao$^{\dag \ddag}$}
\address{$^{\dag}$Department of Mathematics, Faculty of Science and Technology, Tokyo University of Science, 2641 Yamazaki, Noda, Chiba, 278-8510, Japan}
\address{$^{\ddag}$Faculty of Electrical Engineering, The University of Danang - University of Science and Technology, 54 Nguyen Luong Bang, Da Nang, Vietnam}
\email{ngocgiao185@gmail.com}

\begin{document}

\begin{abstract}
Plane Cremona maps of degree four are divided into two different classes according to their homaloidal type: either they are de \Jonq, that is, they have a base point of multiplicity 3, or they are not de \Jonq.
The maps of the former class have length one, while the maps of the latter class have length two, according to the definition of length of a plane Cremona map given by Blanc and Furter.

In this paper, we classify plane Cremona maps of degree four and length two, up to the action of automorphisms of the projective plane on the left-hand side and on the right-hand side of the map.
This classification also allows us to compute the ordinary quadratic length of all such maps.
In particular, we show that this ordinary quadratic length is at most 7 and we classify those maps which attain this upper bound.
\end{abstract}

\maketitle


\section{Introduction}

We work over an algebraically closed field $\K$ of characteristic zero. For convenience, we write $\K^{*}$ to represent $\K \setminus \{0\}$ and $\K^{**}$ to represent $\K \setminus \{0,1\}$.

We denote by $\PP^2$ the projective plane over $\K$ and by $\Bir(\PP^2)$ the \emph{plane Cremona group}, that is, the group of birational maps $\PP^2\dasharrow\PP^2$.
The Noether-Castelnuovo Theorem says that $\Bir(\PP^2)$ is generated by the automorphisms of $\PP^2$ and the quadratic transformation
\begin{equation}\label{sigma_eq}
\sigma\colon \PP^2 \dasharrow \PP^2,
\qquad [x : y : z] \mapsto [yz : xz : xy].
\end{equation}
Consequently, any plane Cremona map $\gamma$ can be decomposed as
\[
\gamma=\alpha_n \circ \sigma \circ \alpha_{n-1} \circ \sigma \circ \cdots \circ \alpha_1 \circ \sigma \circ \alpha_0,
\]
where $\alpha_0, \ldots, \alpha_n$ are automorphisms of the plane, and we can assume that $n$ is minimal under this property.
This $n$ is then called the \emph{ordinary quadratic length} of $\gamma$.
Recall that, in \cite{B-F}, Blanc and Furter defined and studied the \emph{length} of a plane Cremona map $\gamma$ as the minimum number of de \Jonq\ maps needed to decompose $\gamma$.

It is well known that plane Cremona maps of degree four are either de \Jonq\ and they have length one or not de \Jonq\ and they have length two.

In a very recent preprint \cite{CN2}, we introduced and studied the similar notion of \emph{quadratic length} of a plane Cremona map $\gamma$ of degree four, namely the minimal number of (not necessarily ordinary) quadratic maps needed to decompose $\gamma$.
In particular, we proved that plane Cremona maps of degree four and de Jonquières have quadratic length 3, 4, or 5, and we classified those maps that reach the upper bound.

On the other hand, in this paper, we deal with plane Cremona maps of degree four and length two, i.e., which are not de Jonquières, and their ordinary quadratic length.

Let us say that two plane Cremona maps $\varphi,\varphi'\colon\PP^2\dasharrow\PP^2$  are \emph{equivalent} if there exist two automorphisms $\alpha,\alpha'\in\Aut(\PP^2)$ such that
$
\varphi'=\alpha'\circ\varphi\circ\alpha.
$
The classification of equivalence classes of plane Cremona maps of degree two is very well known from the beginning of the study of plane Cremona maps more than one hundred years ago.

Nevertheless, a classification of equivalence classes of plane Cremona maps of degree three has been described only very recently; see \cite{C-D} and \cite{CN}.
In particular, in \cite{CN} we gave a fine classification of them according to their enriched weighted proximity graphs, that is a discrete invariant that encodes the proximity relations among the base points and the possible alignment of three simple base points: we found 31 different such graphs.
Furthermore, concerning their ordinary quadratic length, we proved the following theorem.

\begin{theorem}\label{thm3.1}
Let $\varphi\colon \PP^2 \dashrightarrow \PP^2$ be a plane Cremona map of degree three.
Then, the ordinary quadratic length of $\varphi$ varies between $2$ and $6$. In particular, $\varphi$ has ordinary quadratic length $6$ if and only if it is equivalent to the map
\begin{equation}\label{phi1}
\varphi_1\colon \PP^2 \dashrightarrow \PP^2,
\qquad
[x:y:z] \mapsto [xz^2 +y^3 : yz^2 : z^3].
\end{equation}
\end{theorem}

In this paper we follow the same approach to study equivalence classes of plane Cremona maps of degree four and length two.

Firstly, we classify their possible enriched weighted proximity graph (cf.\ Theorem \ref{thm:123graphs}):  

\begin{theorem}\label{thm:intro1}
There are exactly 123 enriched weighted proximity graphs of plane Cremona maps of degree four and length two, up to isomorphism, listed in Table \ref{table4.3}.
\end{theorem}

Then, for each one of the 123 graphs, say $G$, we show all equivalence classes of plane Cremona maps that have $G$ as enriched weighted proximity graph.

This classification of equivalence classes of plane Cremona maps of degree four and length two will allow us to prove the following theorem.

\begin{theorem}\label{thm:intro2}
Let $\psi\colon \PP^2 \dashrightarrow \PP^2$ be a plane Cremona map of degree four and length two.
Then, the ordinary quadratic length of $\psi$ varies between $2$ and $7$. In particular, $\psi$ has ordinary quadratic length $7$ if and only if it is equivalent to one, and only one, of the following three maps:
\begin{align}
\psi_{1.1} &\colon \PP^2 \dashrightarrow \PP^2,
&&
[x:y:z] \mapsto [z^{2} (x z +y^{2}): x^{2} z^{2}+2 x y^{2} z +y^{4}+y z^{3} : z^{4}]; \\
\psi_{3.1} &\colon \PP^2 \dashrightarrow \PP^2,
&&
[x:y:z] \mapsto [(x z +y^{2})^{2} : z^{2} (x z +y^{2}) : z (x y z +y^{3}+z^{3})]; \\
\psi_{16.2}&\colon \PP^2 \dashrightarrow \PP^2,
&&[x:y:z] \mapsto [yz(xz+y^2) : y^2z^2 : x^2z^2+2xy^2z+y^4+yz^3 ].
\end{align}
\end{theorem}

Let us briefly explain the contents of this paper.

In Section \ref{sec:prel} we recall the definition and properties of infinitely near points and we review the definition of standard coordinates that we introduced in \cite{CN} to describe them.
In Section \ref{sec:plane_Cremona_maps} we recall the basic definitions regarding plane Cremona maps, their homaloidal type and their length and ordinary quadratic length.
In Section \ref{sec:weighted} we define and classify enriched weighted proximity graphs of plane Cremona maps of degree 4 and length 2.
In Section \ref{sec:classification} we give a classification of plane Cremona maps of
degree 4 and length 2, up to automorphisms.
In Section \ref{sec:height} we recall the definition of the height of a point with respect to a plane Cremona map $\gamma$ and we relate it with the ordinary quadratic length of $\gamma$.
In Section \ref{sec:oql} we prove Theorem \ref{thm:intro2}.

\section{Notation, preliminaries and infinitely near points}
\label{sec:prel}

Let us recall the definition of the bubble space of $\PP^2$, useful to define infinitely near points and related properties.

\begin{definition}[{cf.\ \cite[\S7.3.2]{Dolgachev}, \cite[Chapter 5, \S35]{M}}]
Let us denote by $\BB(\PP^2)$ the so-called \emph{bubble space} of $\PP^2$, which is defined as follows.
Consider all surfaces $X$ \emph{above} $\PP^2$, i.e.\ all surfaces $X$ such that there exists a birational morphism $X\to\PP^2$.
If $X_1,X_2$ are two surfaces above $\PP^2$, say $\pi_1\colon X_1\to\PP^2$ and $\pi_2\colon X_2\to\PP^2$ are birational morphisms,
one identifies $p_1\in X_1$ with $p_2\in X_2$ if the birational map $(\pi_2)^{-1}\circ \pi_1\colon X_1\dasharrow X_2$ is a local isomorphism at $p_1$ that sends $p_1$ to $p_2$.
The bubble space $\BB(\PP^2)$ is the union of all points of all surfaces above $\PP^2$ modulo the equivalence relation generated by these identifications.

For any birational morphism $X\to\PP^2$, there is an injective map $X\to\BB(\PP^2)$, therefore we will identify points of $X$ with their images in $\BB(\PP^2)$.

One says that $p_1\in\BB(\PP^2)$ is \emph{infinitely near} $p_2\in\BB(\PP^2)$, say $p_1\in X_1$ and $p_2\in X_2$, with birational morphisms $\pi_1\colon X_1\to\PP^2$ and $\pi_2\colon X_2\to\PP^2$, if the birational map $(\pi_2)^{-1}\circ\pi_1\colon X_1\dasharrow X_2$ is defined at $p_1$, sends $p_1$ to $p_2$, but is not a local isomorphism at $p_1$. In such a case we write that $p_1 \infnear p_2$.

One moreover says that $p_1$ is \emph{in the first neighbourhood} of $p_2$, or that $p_1$ is \emph{infinitely near $p_2$ of the first order}, if $(\pi_2)^{-1}\circ\pi_1$ corresponds locally to the blow-up of $p_2$.
In such a case we write that $p_1 \infnear_1 p_2$.

If $p_1 \infnear p_2$ then one defines the \emph{infinitesimal order} of $p_1$ with respect to $p_2$ by induction, namely if $p_1 \infnear_1 p_3$ and $p_3 \infnear_k p_2$ for some $k$, then $p_1$ is \emph{infinitely near $p_2$ of order $k+1$}. 

If $p_1\infnear p_2$ and $p_1\in X_1$, then there is a unique irreducible curve $E_2\subset X_1$ which corresponds to the exceptional curve of the blowing-up of $p_2\in X_2$. One says that $p_1$ is \emph{proximate} to $p_2$ if $p_1\in E_2$. In such a case we write that $p_1\prox p_2$.

If $p_1\prox p_2$ and $p_1\infnear_k p_2$ with $k >1$, then we say that $p_1$ is \emph{satellite} to $p_2$ and we write $p_1 \satel p_2$. Otherwise, if $p_1$ is not satellite to $p_2$, then we write that $p_1 \notsatel p_2$.

One says that a point $p\in\PP^2\subset\BB(\PP^2)$ is a \emph{proper point of $\PP^2$}.

If $p_1\infnear p_2 \in \PP^2$ where $p_1 \in X_1$ and $\pi_1 : X_1 \rightarrow \PP^2$ is a birational morphism, we say that a plane curve $C$ passes through $p_1$ if $C$ passes through $p_2$ and the strict transform of $C$ on $X_1$ via $\pi_1$ passes through $p_1$.
\end{definition}

\begin{proposition}[Proximity inequality]\label{pro Proximity_inequalities} Let $\varphi\colon S \rightarrow \PP^2$ be a birational morphism, that is, the composition of the blowing-ups at single points, namely $\varphi=\pi_1\circ\pi_2\circ\cdots\circ\pi_{r-1}\circ\pi_r$, where $\pi_i\colon S_i\to S_{i-1}$ is the blowing up at a point $p_i\in S_{i-1}$, $i=1,\ldots,r$, with $\PP^2=S_0$ and $S=S_r$.
Let $C$ be a plane curve and let $C_i$ be the strict transform of $C$ in $S_i$, $i=1,\ldots,r$. Setting $C_0 =C$ and $m_i = \mult_{p_i}(C_{i-1})$, $i=1, \ldots,r$, one has, for each $j = 1, \ldots, r$, 
\[
m_j \geq \sum_{p_k\prox p_j} m_k.
\]
\end{proposition}
\begin{proof}
See $\S2.2$ in \cite{MAC} or Theorem $3.5.3$, Corollary $3.5.4$ in \cite{ECA}.
\end{proof}

We recall the \emph{standard coordinates} of infinitely near points that we introduced in \cite{CN}.

Let $p_1=[a:b:c] \in \PP^2$. Let us consider three cases:
\begin{itemize}
\item[$(i)$] if $c \neq 0$, then $p_1 =\bigg[\dfrac{a}{c}:\dfrac{b}{c}:1\bigg] = [\overline{a}:\overline{b}:1]$;
\item[$(ii)$] if $c = 0$ and $b \neq 0$, then   $p_1 =\bigg[\dfrac{a}{b}:1:0\bigg] = [\overline{a}:1:0]$;
\item[$(iii)$] if $c=b= 0$, then $p_1 = [1:0:0]$.
\end{itemize}

In case $(i)$, we work on the affine chart $U_2 \simeq \K^2_{\overline{x},\overline{y}}$, so that $p_1$ corresponds to the point $\overline{p}_1 = (\overline{a},\overline{b})$, and we define the isomorphism $\alpha_1 \colon \K^2_{\overline{x},\overline{y}} \rightarrow \K^2_{x_0,y_0}$ by
$$\alpha_1(\overline{x},\overline{y}) = (\overline{x}-\overline{a},\overline{y}-\overline{b}).$$

In case $(ii)$, we work on the affine chart $U_1 \simeq \K^2_{\overline{x},\overline{z}}$, so that $p_1$ corresponds to the point $\overline{p}_1 = (\overline{a},0)$, and we define the isomorphism $\alpha_1 \colon \K^2_{\overline{x},\overline{z}} \rightarrow \K^2_{x_0,y_0}$ by
$$\alpha_1(\overline{x},\overline{z}) = (\overline{x}-\overline{a},\overline{z}).$$

In case $(iii)$, we work on the affine chart $U_0 \simeq \K^2_{\overline{y},\overline{z}}$, so that $p_1$ corresponds to the point $\bar p_1=(0,0)$, and we define the isomorphism $\alpha_1 \colon \K^2_{\overline{y},\overline{z}} \rightarrow \K^2_{x_0,y_0}$ by
$$\alpha_1(\overline{y},\overline{z}) = (\overline{y},\overline{z}).$$

In all three cases, we defined $\alpha_1$ in such a way that $\alpha_1(\overline{p}_1) = (0,0)\in \K^2_{x_0,y_0}$.

We blow-up $\K^2_{x_0,y_0}$ at $(0,0)$ and we consider the first chart $\K^2_{x_1,y_1}$ where the blowing-up map is given in coordinates by $x_0=x_1, y_0=x_1y_1$.

In this chart, the exceptional curve $E_1$ has local equation $x_1=0$, hence a point $p_2 \infnear_1 p_1$ corresponds either to the point $(0,t_2) \in E_1$ with $t_2 \in \K$ or to the point which is the origin of the second chart.
In the former case, let us say that $p_2$ has standard coordinates $p_2=(p_1,t_2)$, while in the latter case let us say that $p_2$ has standard coordinates $p_2=(p_1,\infty)$.
Setting $\PP^1=\K\cup\{\infty\}$, let us denote the standard coordinates of $p_2$  by $p_2=(p_1,t_2)$ with $t_2\in\PP^1$.

We go on by blowing-up at $p_2=(p_1,t_2)$, with $t_2\in\PP^1=\K\cup\{\infty\}$.
Either $t_2\in\K$ or $t_2=\infty$. In the former case, with notation as above, let $\alpha_2 \colon \K^2_{x_1,y_1} \rightarrow \K^2_{\bar x_1,\bar y_1}$ be the isomorphism defined by
$$\alpha_2(x_1,y_1) = (x_1, y_1-t_2).$$
In the latter case, $p_2$ corresponds to the origin of the second chart of the blowing-up of $\K^2_{x_0,y_0}$ at $(0,0)$ that we write  $\K^2_{x'_1,y'_1}$, where the blowing-up map is given by $x_0=x'_1 y'_1, y_0=y'_1$. 
Let $\alpha_2 \colon \K^2_{x'_1,y'_1} \rightarrow \K^2_{\bar x_1,\bar y_1}$ be the isomorphism
$$\alpha_2(x'_1,y'_1) = (y'_1, x'_1).$$
In this way, in both cases, in $\K^2_{\bar x_1,\bar y_1}$ the exceptional curve $E_1$ has local equation $\bar x_1=0$ and the point $p_2$ corresponds to the origin $(0,0)$.

We blow-up $\K^2_{\bar x_1,\bar y_1}$ at $(0,0)$ and we consider the first chart $\K^2_{x_2,y_2}$ where the blowing-up map is given in coordinates by $\bar x_1=x_2, \bar y_1=x_2y_2$.
In this chart, the exceptional curve $E_2$ has local equation $x_2=0$, hence a point $p_3 \infnear_1 p_2$ corresponds either to the point $(0,t_3) \in E_2$ with $t_3 \in \K$ or to the point which is the origin of the second chart.

Let us say that $p_3$ has standard coordinates $p_3=(p_1,t_2,t_3)$, where either $t_3\in\K$ in the former case or $t_3=\infty$ in the latter case.

Note that the strict transform of $E_1$ can be seen only in the second chart and it meets $E_2$ at the origin of the second chart.
In other words, the point with standard coordinates $(p_1,t_2,\infty)$ is satellite to $p_1$.

More generally, let us proceed by induction on the infinitesimal order.
Suppose that we have blown-up the point $p_{r-1}$ with standard coordinates $p_{r-1}=(p_1,t_2,\ldots,t_{r-1})$, with $t_i\in\PP^1=\K\cup\{\infty\}$, $i=2,\ldots,r-1$.
Following the procedure described above, we may assume that $p_{r-1}$ is the origin of a chart $\K^2_{\bar x_{r-1},\bar y_{r-1}}$ in such a way that the exceptional curve $E_{r-1}$ has local equation $\bar x_{r-1}=0$.

In the first chart of the blowing up of $\K^2_{\bar x_{r-1},\bar y_{r-1}}$ at $(0,0)$, given in coordinates by $\bar x_{r-1}=x_r, \bar y_{r-1}=x_ry_r$, the exceptional curve $E_r$ has local equation $x_r=0$, hence a point $p_r \infnear_1 p_{r-1}$ corresponds either to the point $(0, t_r) \in E_r$ with $t_r\in\K$ or to the point which is the origin of the second chart, given in coordinates by $\bar x_{r-1}=x_ry_r, \bar y_{r-1}=y_r$.

Let us say that $p_r$ has standard coordinates $p_r=(p_1,t_2,\ldots,t_r)$, where $t_r\in\K$ in the former case and $t_r=\infty$ in the latter case.

In such a way, we are able to define the standard coordinates for each point of $\BB(\PP^2)$ as given by a point in $\PP^2$ and a sequence of points of $\PP^1$.

\begin{remark}\label{rm no_conic1}\label{rek no_conic}
If $p_1 \in \mathbb{P}^2, p_3 \infnear_1 p_2 \infnear_1 p_1$ and $p_3 \odot p_1$, i.e.\ $p_3 \dashrightarrow p_1$, then there is no smooth plane curve passing through ($p_1$, $p_2$ and) $p_3$, because of the proximity inequality at $p_1$.

For the same reason, if $p_1 \in \mathbb{P}^2$, $p_2 \infnear_1 p_1$, $p_3 \infnear_1 p_1$ and $p_2 \neq p_3$, then there is no smooth plane curve passing through ($p_1$ and) $p_2$ and $p_3$.

Moreover, if $p_1 \in \mathbb{P}^2, p_3 \infnear_1 p_2 \infnear_1 p_1$ and $p_1,p_2,p_3$ are collinear, namely $p_3$ lies on the strict transform of the line passing through $p_1$ and $p_2$, then there is no irreducible conic passing through $p_1,p_2,p_3$.
\end{remark}

The following lemma has been proved in \cite{CN}.

\begin{lemma}
Suppose that $p_1, p_2, p_3, p_4, p_5 \in \BB(\PP^2)$ are such that no three of them are collinear and either
\begin{itemize}
\item $p_1,p_2,p_3,p_4,p_5 \in \PP^2$; or
\item $p_1,p_2,p_3,p_4 \in \PP^2$ and $p_5 \infnear_1 p_1$; or
\item $p_1,p_2,p_3 \in \PP^2$, $p_5 \infnear_1 p_4 \infnear_1 p_1$ and $p_5 \notsatel p_1$; or
\item $p_1,p_2,p_3 \in \PP^2$, $p_4 \infnear_1 p_1$, and $p_5 \infnear_1 p_2$; or
\item $p_1,p_2 \in \PP^2$, $p_5 \infnear_1 p_3 \infnear_1 p_1$, $p_4 \infnear_1 p_2$ and $p_5 \notsatel p_1$; or
\item $p_1,p_2 \in \PP^2$, $p_5 \infnear_1 p_4 \infnear_1 p_3 \infnear_1 p_1$, $p_4 \notsatel p_1$, and $p_5 \notsatel p_3$; or
\item $p_5 \infnear_1 p_4 \infnear_1 p_3 \infnear_1 p_2 \infnear_1 p_1 \in \PP^2$, $p_3 \notsatel p_1$, $p_4 \notsatel p_2$, and $p_5 \notsatel p_3$.
\end{itemize}
Then, there exists a unique irreducible plane conic passing through $p_1,\ldots, p_5$.
\end{lemma}

\section{Plane Cremona maps and their lengths}\label{sec:plane_Cremona_maps}

A plane Cremona map $\varphi\colon \PP^2 \dasharrow \PP^2$ can be written as
\begin{equation*}
\varphi( [x : y : z] ) = [f_0(x, y, z) : f_1(x, y, z) : f_2(x, y, z)]
\end{equation*}
where $f_i\in\K[x,y,z]$, $i=0,1,2$, are homogeneous polynomials of the same degree, say $d$, that is called the \emph{degree} of $\varphi$ if $f_0, f_1, f_2$ have no common factor.

Plane Cremona maps of degree 1 are (linear) automorphisms of $\PP^2$, i.e.\ elements of $\Aut(\PP^2)\cong\PGL(3, \K)$.
Plane Cremona maps of degree 2 (3, 4, resp.) are called \emph{quadratic} (\emph{cubic}, \emph{quartic}, resp.).

\begin{definition}
Let us say that  two plane Cremona maps $\varphi,\varphi'\colon\PP^2\dasharrow\PP^2$  are \emph{equivalent} if there exist two automorphisms $\alpha,\alpha'\in\Aut(\PP^2)$ such that
\[
\varphi'=\alpha'\circ\varphi\circ\alpha.
\]
\end{definition}

\begin{remark}
It is very well-known that quadratic Cremona maps are equivalent to one and only one of the following three maps: $\sigma$, $\rho$ and $\tau$; where $\sigma$ is given in \eqref{sigma_eq},
\begin{equation}\label{rhotau}
\rho([x:y:z])=[xy : z^2 : yz];
\qquad
\tau([x:y:z])=[x^2 : xy : y^2-xz].
\end{equation}
Recall that a plane Cremona map of degree two is called:
\begin{itemize}
    \item \emph{ordinary} (or \emph{of the first type}) if it is equivalent to $\sigma$;
    \item \emph{of the second type} if it is equivalent to $\rho$;
    \item \emph{of the third type} if it is equivalent to $\tau$.
\end{itemize}
\end{remark}

Recall that the linear system defining a plane Cremona map $\varphi$ of degree $d$ has base points $p_1,\ldots,p_r\in\BB(\PP^2)$, of respective multiplicity $m_1,\ldots,m_r$,
that satisfy the following equations:
\begin{equation}\label{eq:Cremona}
d^2-1=\sum_{i=1}^r m_i^2,
\qquad\qquad
3(d-1)=\sum_{i=1}^r m_i.
\end{equation}
One says that $(d;m_1,\ldots,m_r)$ is the homaloidal type of $\varphi$.

A plane Cremona map is called \emph{de Jonqui\`eres} if it has degree $d$ and a base point of multiplicity $d-1$, i.e.\ its homaloidal type is $(d;d-1,1,\ldots,1)=(d;d-1,1^{2d-2})$.

According to the Noether-Castelnuovo Theorem, any plane Cremona map $\varphi\colon\PP^2\dasharrow\PP^2$ can be written
\begin{equation}\label{eq:decomp}
\varphi = \alpha_n \circ \sigma \circ \alpha_{n-1} \circ \sigma \circ \cdots \circ \alpha_1 \circ \sigma \circ \alpha_0
\end{equation}
where $\alpha_i\in\Aut(\PP^2)$, $i=0,\ldots,n$, for some integer $n$.

\begin{definition}
Let us call \eqref{eq:decomp} a \emph{decomposition} of $\varphi$.
Let us say that a decomposition \eqref{eq:decomp} is \emph{minimal} if $n$ is minimal among all decompositions of $\varphi$.
Let us call such $n$ the \emph{ordinary quadratic length} of $\varphi$ and let us denote it by $\oq(\varphi)$.
\end{definition}

Equivalently, the ordinary quadratic length of a plane Cremona map $\varphi$ of degree $\geq2$ is the minimum $n$ such that there exist ordinary quadratic maps $\psi_1,\psi_2,\ldots,\psi_n$ with
\begin{equation}\label{eq:decomp2}
\varphi=\psi_n\circ\psi_{n-1}\circ \cdots \circ \psi_2\circ \psi_1.
\end{equation}

\begin{definition}
Let us call the \emph{quadratic length} of plane Cremona map $\varphi$ the minimum $n$ such that there exists a decomposition \eqref{eq:decomp2} where $\psi_i$ is a (not necessarily ordinary) quadratic map, for each $i=1,\ldots,n$, and denote it by $\q(\varphi)$.
\end{definition}

Recall that Blanc and Furter in \cite{B-F} defined the \emph{length} of a plane Cremona map $\varphi$ as the minimum $n$ such that there exists a decomposition \eqref{eq:decomp2} where $\psi_i$ is a de Jonqui\`eres map, for each $i=1,\ldots,n$, and denoted it by $\lgth(\varphi)$.
Clearly, one has that
\begin{equation*}
\lgth(\varphi) \leq \q(\varphi) \leq \oq(\varphi).
\end{equation*}
Clearly, two equivalent plane Cremona maps have the same length, quadratic length and ordinary quadratic length.

\begin{examples}
The quadratic maps $\rho$ and $\tau$ defined in \eqref{rhotau} trivially have $\q(\rho)=\q(\tau)=1$, whereas it is well known that $\oq(\rho)=2$ and $\oq(\tau)=4$, cf.\ \cite{C-D}, \cite{CN}.
\end{examples}

Note that Equations \eqref{eq:Cremona} imply that plane Cremona maps of degree two and three are de Jonqui\`eres, whereas those of degree four either are de Jonqui\`eres or their homaloidal type is $(4;2,2,2,1,1,1)=(4;2^3,1^3)$.
In the former case, they have length one, while in the latter case, they have length two. In this paper we deal with those of length two.

\section{Enriched weighted proximity graphs}\label{sec:weighted}

\begin{definition}\label{def weighted_proximity_graph}
Let $\varphi$ be a plane Cremona map. Let us associate to $\varphi$ a \emph{weighted digraph} $G_\varphi$, called the \emph{weighted proximity graph of (the base points of) $\varphi$}, defined as follows:
\begin{itemize}
\item the vertices of $G_\varphi$ are the base points $p_1,\ldots,p_r \in \BB(\PP^2)$ of $\varphi$;
\item there is an arrow $p_i \to p_j$ if and only if $p_i$ is proximate to $p_j$;
\item each vertex $p_i$ is weighted with the multiplicity $m_i=\mult_{p_i}(\varphi)$ of $\varphi$ at $p_i$.
\end{itemize}
\end{definition}

It is easy to check that (weighted) proximity graphs have the property of being \emph{admissible}, according to the following definition.

\begin{definition}
Let us say that a digraph is \emph{admissible} if it is acyclic and satisfies the following three properties:
\begin{enumerate}[$(i)$]
\item each vertex has outdegree at most 2;
\item if a vertex $u$ has outdegree 2, say $u \to v$ and $u \to w$, then either $v \to w$ or $w \to v$;
\item fixing two vertices $v$ and $w$, there exists at most one vertex $u$ such that $u \to v$ and $u \to w$.
\end{enumerate}
\end{definition}

\begin{remark}\label{rem:weighted1}
Weighted proximity graphs satisfy proximity inequalities \ref{pro Proximity_inequalities}, namely the weight $m_j$ of a vertex $p_j$ is such that
\[
m_j \geq \sum_{p_i \to p_j} m_i
\]
where $m_i$ is the weight of the vertex $p_i$.
\end{remark}

\begin{example}
Let $\sigma$, $\rho$ and $\tau$ be the quadratic maps defined in \eqref{sigma_eq} and \eqref{rhotau}. Their respective proximity graphs $G_\sigma$, $G_\rho$ and $G_\tau$ are:
\begin{center}
$G_\sigma$ =
\scalebox{0.4}{%
\begin{tikzpicture}[->,>=stealth',shorten >=2pt,auto,node distance=2.5cm,ultra thick,baseline=-1.25ex]
  \tikzstyle{every state}=[fill=white,draw=black,text=black]
  \node[state,draw=red,text=red] (A)                    {\huge 1};
  \node[state,draw=red,text=red] (B) [right of=A] {\huge 1};
  \node[state,draw=red,text=red] (C) [right of=B] {\huge 1};
  \end{tikzpicture}%
} \qquad
$G_\rho$ =
\scalebox{0.4}{%
\begin{tikzpicture}[->,>=stealth',shorten >=2pt,auto,node distance=2.5cm,ultra thick,baseline=-1.25ex]
  \tikzstyle{every state}=[fill=white,draw=black,text=black]

  \node[state,draw=red,text=red] (A)                    {\huge 1};
  \node[state] (B) [right of=A] {\huge 1};
  \node[state,draw=red,text=red] (C) [right of=B] {\huge 1};
  \path (B) edge (A);
\end{tikzpicture}%
} \qquad
$G_\tau$ =
\scalebox{0.4}{%
\begin{tikzpicture}[->,>=stealth',shorten >=2pt,auto,node distance=2.5cm,ultra thick,baseline=-1.25ex]
  \tikzstyle{every state}=[fill=white,draw=black,text=black]

  \node[state,draw=red,text=red] (A)                    {\huge 1};
  \node[state] (B) [right of=A] {\huge 1};
  \node[state] (C) [right of=B] {\huge 1};

\path (C) edge (B);
\path (B) edge (A);
\end{tikzpicture}%
}
\end{center}
\end{example}

\begin{notation}
When we draw the weighted proximity graph of a plane Cremona map $\varphi$, for readers' convenience we write in \textcolor{red}{red} the vertices corresponding to \emph{proper base points} and in \textcolor{black}{black} the vertices corresponding to \emph{infinitely near points}.
\end{notation}

\begin{remark}
Clearly, two equivalent plane Cremona maps have the same weighted proximity graph. The converse is true for quadratic maps but it is not true for cubic, or higher degree, maps, cf.\ \cite{CN}. 
\end{remark}

In \cite{CN} it has been proved that there are exactly 21 weighted proximity graphs of plane Cremona maps of degree three.

We want to classify weighted proximity graphs of plane Cremona maps of degree four and length two.
Note that they have six vertices: three of them with weight 2 and the other three vertices of weight 1.

\begin{theorem}\label{thm:graphsndJ}
There are exactly 53 admissible graphs having six vertices with weights respectively $2, 2, 2, 1, 1, 1$, up to isomorphism, that are listed in Table \ref{table4.2a}.
\end{theorem}






\begin{proof}
The proximity inequalities \ref{rem:weighted1} imply that
\begin{itemize}
\item the indegree of a vertex with weight 1 is at most one;
\item the indegree of a vertex $u$ with weight 2 is at most two; if it is two, it means that the two arrows ending in $u$ are starting from two vertices of weight 1;
\item the outdegree of a vertex $u$ with weight 2 is at most one; if it is one, the arrow starting from $u$ must end in a vertex of weight 2;
\item the outdegree of a vertex $v$ with weight 1 is at most two.
\end{itemize}

Taking into account these properties, we want to list all weighted admissible graphs according to the number of arrows of the graph.

Let us describe the construction of the graphs in an algorithmic way.

\begin{itemize}
\item[Step 1:] set the number $N$ of arrows to be 0;
\item[Step 2:] let $L_N$ be the list of graphs with $N$ arrows;
\item[Step 3:] define an empty list $L'$;
\item[Step 4:] for each graph $G$ in $L_N$, let $G'$ be a graph obtained from $G$ by adding one arrow;
\item[Step 5:] check if $G'$ is admissible and it satisfies the proximity inequalities;
\item[Step 6:] if the answer to Step 5 is affirmative, check if $G'$ is isomorphic to any graph in the list $L'$;
\item[Step 7:] if the answer to Step 6 is negative, add $G'$ to the list $L'$;
\item[Step 8:] repeat Steps 4, 5, 6 and 7 for all possible arrows;
\item[Step 9:] increase $N$ by 1, set $L_N=L'$ and, if $N \leq 6$ go back to Step 3.
\end{itemize}

In Step 2, we set $L_0$ to be the list with only the weighted admissible graph with no arrow, that is number 53 in Table \ref{table4.2a}.

Running the algorithm, one finds that the list $L_1$ of graphs with exactly one arrow contains the graphs 50, 51 and 52.
Then, the list $L_2$ contains the graphs 37, 38, 43, 44, 45, 46, 47, 48 and 49.
Hence, the list $L_3$ contains the graphs 21, 23, 24, 25, 27, 28, 30, 31, 32, 34, 35, 36, 39, 40, 41 and 42.
The list $L_4$ contains the graphs 5, 6, 8, 9, 11, 12, 13, 15, 16, 17, 19, 20, 22, 26, 29 and 33.
The list $L_5$ contains the graphs 2, 3, 4, 7, 10, 14 and 18.
Finally, the list $L_6$ contains only graph 1.
\end{proof}

Recall that in \cite{CN} we enriched the weighted proximity graphs of plane Cremona maps of degree three by adding the line passing through three base points:
we found 31 non-isomorphic enriched weighted proximity graphs of plane Cremona maps of degree three.

\begin{remark}\label{rem:line}
Recall that the general member of the linear system of plane curves defining a plane Cremona map $\varphi$ is irreducible. Hence, B\'ezout Theorem implies that, if $\varphi$ has degree four and length two, then:
\begin{itemize}
  \item there is no conic through the six base points of $\varphi$;
  \item there is no line through four base points of $\varphi$;
  \item there is no line through the three base points of multiplicity 2 of $\varphi$;
  \item there is no line through two base points of multiplicity 2 and another base point.
\end{itemize}
\end{remark}

The previous remark implies that we may enrich a weighted admissible graph with information about lines through three vertices, where either all vertices have weight 1 or one vertex has weight 2 and the other two vertices have weight 1.

\begin{remark}\label{rem:line2}
The proximity inequalities imply that, if a line $\ell$ passes through a vertex $j$ and there is an arrow from the vertex $j$ to the vertex $i$, then the line $\ell$ must pass also through the vertex $i$.
\end{remark}

\begin{definition}
Let us add to the weighted proximity graph $G_\varphi$ of a plane Cremona map $\varphi$ of degree four and length two the list of lines passing through three base points of $\varphi$.
Let us call this object the \emph{enriched weighted proximity graph} of $\varphi$.
\end{definition}

In next theorem we classify all possible enriched weighted admissible graphs, up to isomorphism.
In Table \ref{table4.3}, there is the list of 123 enriched weighted admissible graphs. Each graph is listed as $n.m$, where $n$ is the number of the weighted admissible graph in Table \ref{table4.2a} and $m\geq1$ is such that the case $m=1$ corresponds to the graph with no line through the vertices, while each case with $m\geq2$ has one or more lines marked with $(i, j, k)$, if there is a line passing through the vertices $p_i, p_j, p_k$, where $\{i,j,k\}\subset\{1,2,\ldots,6\}$ and the vertices $p_1,\ldots,p_6$ are listed from the left hand side to the right hand side.

We will then show that each enriched weighted admissible graph is realized by some plane Cremona map of degree 4 and length 2.
More precisely, one finds that all plane Cremona maps with a given enriched weighted proximity graph depend on the number of parameters listed in the last column of Table \ref{table4.3}.

The fourth column of Table \ref{table4.3} in the row $n.m$ lists the type of the plane Cremona map that is the inverse of type $n.m$.

The third column of Table \ref{table4.3} lists the ordinary quadratic length of plane Cremona maps of that type, that is studied in Section \ref{sec:oql}.

\begin{theorem}\label{thm:123graphs}
There are exactly 123 enriched weighted proximity graphs of plane Cremona maps of degree four and length two, up to isomorphism, listed in Table \ref{table4.3}.
\end{theorem}




\begin{proof}
In Table \ref{table4.3}, we set the enriched weighted admissible graph of type $n.1$ as the graph of type $n$ in Table \ref{table4.2a} with no line.
Then, we add a line through three vertices, according to Remarks \ref{rem:line} and \ref{rem:line2}.
For example, this implies that the weighted admissible graphs of type $n=1,\ldots,10$ cannot have any line through three points.

Moreover, it follows that only type 53 may have three different lines and there is a unique enriched weighted admissible graph of type 53 with three lines, up to isomorphism.

In order to list all possible enriched weighted admissible graphs of type $n.m$, $m\geq 2$, we follow an algorithm.

\begin{itemize}
\item[Step 1:] let $G_n$ be the weighted admissible graph of type $n$ in Table \ref{table4.2a}.
\item[Step 2:] for all subsets $\{i,j,k\}$ with $i < j < k$ of $\{1,2,\ldots,6\}$, check if the vertices $p_i, p_j, p_k$ satisfy the conditions in Remarks \ref{rem:line} and \ref{rem:line2}.
\item[Step 3:] if the answer to Step 2 is affirmative, check if two different subsets $\{i,j,k\}$ and $\{i',j',k'\}$ are such that there is an automorphism of the graph sending one subset into the other one.
\item[Step 4:] answering to Step 3, we are able to list all enriched weighted admissible graphs of type $n$ with one line.
\item[Step 5:] if it is possible, add another line to the graphs obtained in Step 4 and check if two such enriched graphs are isomorphic.
\item[Step 6:] answering to Step 5, we are able to list all enriched weighted admissible graphs of type $n$ with two lines.
\end{itemize}


It remains to show that, for each enriched weighted admissible graph $G$ in Table \ref{table4.3}, there exists a plane Cremona map with $G$ as enriched weighted proximity graph.
This is done in next Table \ref{table4.4}.
\end{proof}

\begin{longtable}[c]{|c|c|}
\caption{List of 123 quartic maps.}
\label{table4.4}
\endfirsthead
\endhead
\hline

$\sharp$    & Map \\

\hline
1.1 & $[z^{2} \left(x z +y^{2}\right):
x^{2} z^{2}+2 x \,y^{2} z +y^{4}+y \,z^{3}: z^{4}]
$ \\
\hline
2.1 & $[\left(x z +y^{2}+y z \right) \left(x z +y^{2}\right): 
z^{2} \left(x z +y^{2}\right): z^{4}]
$ \\
\hline
3.1 & $[\left(x z +y^{2}\right)^{2}: z^{2} \left(x z +y^{2}\right): 
z \left(x y z +y^{3}+z^{3}\right)]
$ \\
\hline
4.1 & $[x^{2} z^{2}+2 x \,y^{2} z +x y \,z^{2}+y^{4}: 
y z \left(x z +y^{2}\right): y^{2} z^{2}]
$ \\
\hline
5.1 & $[-\left(x z -y^{2}\right) \left(x z +y^{2}\right): 
y z \left(x z +y^{2}\right): y^{2} z^{2}]
$ \\
\hline
6.1 & $[\left(x z +y^{2}\right)^{2}: y z \left(x z +y^{2}\right): 
z \left(x^{2} z +x \,y^{2}+y^{2} z \right)]
$ \\
\hline
7.1 & $[x y \,z^{2}: \left(x z +y^{2}\right)^{2}: y^{3} z]$ \\
\hline
\end{longtable}


\section{Classification results for quartic non-de \Jonq\ maps}\label{sec:classification}

In this section we give a classification of equivalence classes of plane Cremona maps of degree 4 and length 2.

\begin{theorem}\label{thm:classification}
Let $\varphi_{n.m}$ be a plane Cremona map of type $n.m$ listed in Table \ref{table4.3}.
Then $\varphi_{n.m}$ is equivalent to a map of the corresponding type in Table \ref{table4.4}.
\end{theorem}

\begin{proof}
Let $\varphi_{n.m}$ be as above. Then, there exists an automorphism $\alpha$ of $\PP^2$ such that $\alpha\circ \varphi_{n.m}$ has the base points with the standard coordinates in the following list, where each item corresponds to the type in brackets:
\begin{itemize}

\item[(1.1)]  $p_1=[1:0:0]$, $p_2=(p_1,0)$, $p_3=(p_1,0,-1)$, $p_4=(p_1,0,-1,0)$, $p_5=(p_1,0,-1,0,\infty)$ and $p_6=(p_1,0,-1,0,\infty,1)$; since $p_1, p_2, p_3, p_4, p_5$ are uniquely determined by $p_6$, according to the enriched weighted proximity graph, in what follows we may and will write only $p_6$;

\item[(2.1)]  $p_4=([1:0:0],0,-1,1)$ and $p_6=([1:0:0],0,-1,0,0)$;
\item[(3.1)]  $p_6=([1:0:0],0,-1,0,0,1)$;
\item[(4.1)]  $p_1=[0:0:1]$ and $p_6=([1:0:0],0,-1,\infty,-1)$;
\item[(5.1)]  $p_1=[0:0:1]$, $p_4=([1:0:0],0,1)$ and $p_6=([1:0:0],0,-1,0)$;
\item[(6.1)]  $p_1=[0:0:1]$ and $p_6=([1:0:0],0,-1,0,1)$;  
\item[(7.1)]  $p_2=([0:0:1],\infty)$ and $p_6=([1:0:0],0,-1,\infty)$;
\item[(8.1)]  $p_2=([0:0:1],\infty)$, $p_5=([1:0:0],0,-1)$ and $p_6=([1:0:0],0,t)$, where $t\in\K^*$, $t\ne-1$;
\item[(9.1)] $p_2=([0:0:1],\infty)$ and $p_6=([1:0:0],0,-1,1)$;
\item[(10.1)] $p_2=([0:0:1],\infty)$ and $p_6=([1:0:0],0,\infty,-1)$;
\item[(11.1)] $p_2=([0:0:1],\infty)$, $p_4=([1:0:0],0)$ and $p_6=([1:0:0],-1,1)$;
\item[(11.2)] $p_2=([0:0:1],\infty)$, $p_4=([1:0:0],0)$ and $p_6=([1:0:0],-1,0)$;
\item[(12.1)] $p_2=([0:0:1],\infty)$ and $p_6=([1:0:0],0,-1,1)$;
\item[(12.2)] $p_2=([0:0:1],\infty)$ and $p_6=([1:0:0],0,0,-1)$;
\item[(13.1)] $p_3=([0:0:1],\infty,-1)$ and $p_6=([1:0:0],0,t)$, where $t\in\K^*$, $t\ne-1$;
\item[(13.2)] $p_3=([0:0:1],\infty,0)$ and $p_6=([1:0:0],0,-1)$;
\item[(14.1)] $p_3=([0:0:1],\infty,\infty)$ and $p_6=([1:0:0],0,-1)$;
\item[(15.1)] $p_2=([0:0:1],\infty)$, $p_3=([0:0:1],-1)$ and $p_6=([1:0:0],0,-1)$;
\item[(16.1)] $p_3=([1:0:0],0,-1)$ and $p_6=([0:0:1],\infty,t)$, where $t\in\K^*$, $t\ne-1$;
\item[(16.2)] $p_3=([1:0:0],0,-1)$ and $p_6=([0:0:1],0,-1)$;
\item[(16.3)] $p_3=([1:0:0],0,-1)$ and $p_6=([0:0:1],\infty,0)$;
\item[(17.1)] $p_4=([1:0:0],0,-1,1)$ and $p_6=([0:0:1],\infty)$;
\item[(17.2)] $p_4=([1:0:0],0,-1,1)$ and $p_6=([0:0:1],0)$; 
\item[(18.1)] $p_5=([1:0:0],0,-1,0,\infty)$ and $p_6=[0:0:1]$;
\item[(19.1)] $p_4=([1:0:0],0,-1,0)$, $p_5=([1:0:0],0,-1,1)$ and $p_6=[0:0:1]$; 
\item[(20.1)] $p_5=([1:0:0],0,-1,0,-1)$ and $p_6=[0:0:1]$;
\item[(21.1)] $p_1 = [0: 1: 0]$, $p_2=[0: 0: 1]$ and $p_6=([1:0:0],-1,1,t)$, where $t\in\K$, $t\ne-1$;
\item[(21.2)] $p_1 = [0: 1: 0]$, $p_2=[0: 0: 1]$ and $p_6=([1:0:0],-1,0,-1)$;
\item[(22.1)] $p_1 = [0: 1: 0]$, $p_2=[0: 0: 1]$ and p$_6=([1:0:0],-1,\infty,1)$;
\item[(23.1)] $p_1 = [0: 1: 0]$, $p_2=[0: 0: 1]$, $p_4=([1:0:0],t)$ and $p_6=([1:0:0],-1,1)$, where $t\in\K^*$, $t\ne-1$;
\item[(23.2)] $p_1 = [0: 1: 0]$, $p_2=[0: 0: 1]$, $p_4=([1:0:0],t)$ and $p_6=([1:0:0],-1,0)$, where $t\in\K^*$, $t\ne-1$;
\item[(24.1)] $p_2 = ([0: 0: 1],1)$, $p_4=([0: 1: 0],1)$ and $p_6=([1:0:0],t)$, where $t\in\K^*$, $t\ne-1$;
\item[(25.1)] $p_1 = [0: 1: 0]$, $p_3=([0: 0: 1],1)$ and $p_6=([1:0:0],1,t)$, $t\in\K^*$, $t\ne-1$;
\item[(25.2)] $p_1 = [0: 1: 0]$, $p_3=([0: 0: 1],1)$ and $p_6=([1:0:0],1,0)$;
\item[(26.1)] $p_1 = [0: 1: 0]$, $p_3=([0: 0: 1],1)$ and $p_6=([1:0:0],1,\infty)$;
\item[(27.1)] $p_1 = [0: 1: 0]$, $p_3=([0: 0: 1],1)$, $p_5=([1:0:0],1)$, $p_6=([1:0:0],t)$, with $t\in\K^{**}$;
\item[(28.1)] $p_1 = [0: 0: 1]$, $p_5=([1:0:0],0,1,t)$ and $p_6=[1:1:1]$, where $t\in\K^*$;
\item[(29.1)] $p_1 = [0: 0: 1]$, $p_5=([1:0:0],0,1,\infty)$ and $p_6=[1:1:1]$;
\item[(30.1)] $p_1 = [0: 0: 1]$, $p_4=([1:0:0],0,1)$, $p_5=([1:0:0],0,t)$ and $p_6=[1:1:1]$, where $t\in\K^{**}$;
\item[(31.1)] $p_3 = ([1: 0: 0],0,t)$, $p_5=([0:0:1],0)$ and $p_6=[1:1:1]$, where $t\in\K^{**}$;
\item[(31.2)] $p_3 = ([1: 0: 0],0,1)$, $p_5=([0:0:1],1)$ and $p_6=[1:1:1]$;
\item[(32.1)] $p_2 = ([0: 0: 1],0)$, $p_5=([1:0:0],0,t)$ and $p_6=[1:1:1]$, where $t\in\K^{**}$;
\item[(32.2)] $p_2 = ([0: 0: 1],0)$, $p_5=([1:0:0],0,0)$ and $p_6=[1:1:1]$;
\item[(32.3)] $p_2 = ([0: 0: 1],0)$, $p_5=([1:0:0],1,1)$ and $p_6=[1:1:1]$;
\item[(33.1)] $p_2 = ([0: 0: 1],0)$, $p_5=([1:0:0],0,\infty)$ and $p_6=[1:1:1]$;
\item[(33.2)] $p_2 = ([0: 0: 1],0)$, $p_5=([1:0:0],1,\infty)$ and $p_6=[1:1:1]$;
\item[(34.1)] $p_2 = ([0: 0: 1],0)$, $p_4=([1:0:0],0)$, $p_5=([1:0:0],t)$ and $p_6=[1:1:1]$, where $t\in\K^{**}$;
\item[(34.2)] $p_2 = ([0: 0: 1],0)$, $p_4=([1:0:0],0)$, $p_5=([1:0:0],1)$ and $p_6=[1:1:1]$;
\item[(35.1)] $p_4=([1:0:0],\infty,1,t)$, $p_5=[0:1:0]$ and $p_6=[1:1:1]$, where $t\in\K^*$;
\item[(35.2)] $p_4=([1:0:0],\infty,1,1)$, $p_5=[0:1:0]$ and $p_6=[1:1:0]$;
\item[(36.1)] $p_3=([1:0:0],\infty,t)$, $p_5=([0:1:0],\infty)$ and $p_6=[1:1:1]$, where $t\in\K^{**}$;
\item[(36.2)] $p_3=([1:0:0],\infty,-1)$, $p_5=([0:1:0],\infty)$ and $p_6=[1:1:0]$;
\item[(36.3)] $p_3=([1:0:0],\infty,-1)$, $p_5=([0:1:0],0)$ and $p_6=[1:1:1]$;
\item[(36.4)] $p_3=([1:0:0],\infty,-1)$, $p_5=([0:1:0],1)$ and $p_6=[1:1:1]$;
\item[(37.1)] $p_1 = [0: 0: 1]$, $p_2=[0: 1: 0]$, $p_3=[1: 0: 0]$ and $p_6 =([1: 1: 1],t,u)$, where $t\in\K^{**}$, $u\in\K$, $u\ne t^2-t$;
\item[(37.2)] $p_1 = [0: 0: 1]$, $p_2=[0: 1: 0]$, $p_3=[1: 0: 0]$ and $p_6 =([1: 1: 1],0,u)$, where $u\in\K^*$;
\item[(37.3)] $p_1 = [0: 0: 1]$, $p_2=[0: 1: 0]$, $p_3=[1: 0: 0]$ and $p_6 =([1: 1: 1],t,0)$, where $t\in\K^{**}$;
\item[(38.1)] $p_1 = [0: 0: 1]$, $p_2=[0: 1: 0]$, $p_5=([1: 0: 0],t,u)$ and $p_6 =[1: 1: 1]$, where $t\in\K^{**}$, $u\in\K$, $u\ne t-t^2$;
\item[(38.2)] $p_1 = [0: 0: 1]$, $p_2=[0: 1: 0]$, $p_5=([1: 0: 0],t,0)$ and $p_6 =[1: 1: 1]$, where $t\in\K^{**}$;
\item[(38.3)] $p_1 = [0: 0: 1]$, $p_2=[0: 1: 0]$, $p_5=([1: 0: 0],1,u)$ and $p_6 =[1: 1: 1]$, where $u\in\K^*$;
\item[(39.1)] $p_1 = [0: 0: 1]$, $p_2=[0: 1: 0]$, $p_5=([1: 0: 0],t,\infty)$ and $p_6=[1: 1: 1]$, where $t\in\K^{**}$;
\item[(39.2)] $p_1 = [0: 0: 1]$, $p_2=[0: 1: 0]$, $p_5=([1: 0: 0],1,\infty)$ and $p_6=[1: 1: 1]$; 
\item[(40.1)] $p_1 = [0: 0: 1]$, $p_5=([1: 0: 0],0,t)$ and $p_6=([1: 1: 1],\infty)$, where $t\in\K^*$, $t\ne-1$;
\item[(40.2)] $p_1 = [0: 0: 1]$, $p_5=([1: 0: 0],0,1)$ and $p_6=([1: 1: 1],1)$;
\item[(40.3)] $p_1 = [0: 0: 1]$, $p_5=([1: 0: 0],0,1)$ and $p_6=([1: 1: 1],0)$;
\item[(41.1)] $p_2 = ([1: 0: 0],0)$, $p_4=([0:0:1],\infty)$ and $p_6=([1: 1: 1],t)$, where $t\in\K^{**}$, $t\ne-1/2$;
\item[(41.2)] $p_2 = ([1: 0: 0],0)$, $p_4=([0:0:1],\infty)$ and $p_6=([1: 1: 1],0)$;
\item[(41.3)] $p_2 = ([1: 0: 0],0)$, $p_4=([0:0:1],1)$ and $p_6=([1: 1: 1],\infty)$;
\item[(41.4)] $p_2 = ([1: 0: 0],0)$, $p_4=([0:0:1],\infty)$ and $p_6=([1: 1: 1],1)$;
\item[(41.5)] $p_2 = ([1: 0: 0],0)$, $p_4=([0:0:1],1)$ and $p_6=([1: 1: 1],0)$;
\item[(42.1)] $p_2 = ([1: 0: 0],0)$, $p_3=[0:0:1]$ and $p_6=([1: 1: 1],\infty,t)$, where $t\in\K^*$, $t\ne-1$;
\item[(42.2)] $p_2 = ([1: 0: 0],0)$, $p_3=[0:0:1]$ and $p_6=([1: 1: 1],0,-1)$;
\item[(42.3)] $p_2 = ([1: 0: 0],0)$, $p_3=[0:0:1]$ and $p_6=([1: 1: 1],1,0)$;
\item[(42.4)] $p_2 = ([1: 0: 0],0)$, $p_3=[0:0:1]$ and $p_6=([1: 1: 1],\infty,0)$;
\item[(43.1)] $p_1 = [0: 0: 1]$, $p_3=([1: 0: 0],1)$, $p_5=([0: 1: 0],1)$ and $p_6=[t,u,1]$, where $t,u\in\K^{**}$, $u\ne t/(t-1)$;
\item[(43.2)] $p_1 = [0: 0: 1]$, $p_3=([1: 0: 0],1)$, $p_5=([0: 1: 0],1)$ and $p_6=[t,1,1]$, where $t\in\K^{**}$;
\item[(43.3)] $p_1 = [0: 0: 1]$, $p_3=([1: 0: 0],1)$, $p_5=([0: 1: 0],1)$ and $p_6=[1,1,1]$;
\item[(44.1)] $p_1 = [0: 0: 1]$, $p_2=[0: 1: 0]$, $p_4=([1: 0: 0],t)$, $p_5=([1: 0: 0],u)$ and $p_6=[1,1,1]$, where $t,u\in\K^{**}$, $u\ne t$;
\item[(44.2)] $p_1 = [0: 0: 1]$, $p_2=[0: 1: 0]$, $p_4=([1: 0: 0],t)$, $p_5=([1: 0: 0],1)$ and $p_6=[1,1,1]$, where $t\in\K^{**}$;
\item[(45.1)] $p_1 = [0: 0: 1]$, $p_2=[1: 0: 0]$, $p_4=([0: 1: 0],u/t)$ and $p_6=([1: 1: 1],(u-1)/(u-t))$, where $t,u\in\K^{**}$, $u\ne t,1-t$;
\item[(45.2)] $p_1 = [0: 0: 1]$, $p_2=[1: 0: 0]$, $p_4=([0: 1: 0],t)$ and $p_6=([1: 1: 1],0)$, where $t\in\K^{**}$;
\item[(45.3)] $p_1 = [0: 0: 1]$, $p_2=[1: 0: 0]$, $p_4=([0: 1: 0],\infty)$ and $p_6=([1: 1: 1],t)$, where $t\in\K^{**}$;
\item[(45.4)] $p_1 = [0: 0: 1]$, $p_2=[1: 0: 0]$, $p_4=([0: 1: 0],t)$ and $p_6=([1: 1: 1],\infty)$, where $t\in\K^{**}$;
\item[(45.5)] $p_1 = [0: 0: 1]$, $p_2=[1: 0: 0]$, $p_4=([0: 1: 0],1)$ and $p_6=([1: 1: 1],0)$;
\item[(46.1)] $p_1 = [1: 0: 0]$, $p_4=([0: 1: 0],\infty,1)$, $p_5=[1: 1: 1]$ and $p_6=[t:u:1]$, where $t\in\K^{**}$, $u\in\K$, $u\ne1,t,1/t$;
\item[(46.2)] $p_1 = [1: 0: 0]$, $p_4=([0: 1: 0],\infty,1)$, $p_5=[1: 1: 1]$ and $p_6=[t:1:1]$, where $t\in\K^{**}$;
\item[(46.3)] $p_1 = [1: 0: 0]$, $p_4=([0: 1: 0],\infty,1)$, $p_5=[1: 1: 1]$ and $p_6=[1:u:1]$, where $u\in\K$, $u\ne1$;
\item[(47.1)] $p_2 = ([1: 0: 0],\infty)$, $p_4=([0: 1: 0],\infty)$, $p_5=[1: 1: 1]$ and $p_6=[t:u:1]$, where $t,u\in\K^{**}$, $u\ne 1/t$;
\item[(47.2)] $p_2 = ([1: 0: 0],\infty)$, $p_4=([0: 1: 0],\infty)$, $p_5=[1: 1: 1]$ and $p_6=[t:1:1]$, where $t\in\K^{**}$;
\item[(47.3)] $p_2 = ([1: 0: 0],\infty)$, $p_4=([0: 1: 0],\infty)$, $p_5=[1: 1: 1]$ and $p_6=[0:u:1]$, where $u\in\K^{**}$;
\item[(47.4)] $p_2 = ([1: 0: 0],\infty)$, $p_4=([0: 1: 0],\infty)$, $p_5=[1: 1: 1]$ and $p_6=[1:u:1]$, where $u\in\K^{**}$;
\item[(47.5)] $p_2 = ([1: 0: 0],\infty)$, $p_4=([0: 1: 0],\infty)$, $p_5=[1: 1: 1]$ and $p_6=[0:1:1]$;
\item[(48.1)] $p_2 = ([1: 0: 0],\infty)$, $p_3=[0: 1: 0]$, $p_5=([1: 1: 1],1)$ and $p_6=[t:u:1]$, where $u\in\K^{**}$, $t\in\K$, $t\ne1,u,(2u-1)/u$;
\item[(48.2)] $p_2 = ([1: 0: 0],t)$, $p_3=[0: 1: 0]$, $p_5=([1: 1: 1],0)$ and $p_6=[0:0:1]$, where $t\in\K^{**}$;
\item[(48.3)] $p_2 = ([1: 0: 0],\infty)$, $p_3=[0: 1: 0]$, $p_5=([1: 1: 1],1)$ and $p_6=[t:1:1]$, where $t\in\K$, $t\ne1$;
\item[(48.4)] $p_2 = ([1: 0: 0],t)$, $p_3=[0: 1: 0]$, $p_5=([1: 1: 1],\infty)$ and $p_6=[0:0:1]$, where $t\in\K^{**}$;
\item[(48.5)] $p_2 = ([1: 0: 0],\infty)$, $p_3=[0: 1: 0]$, $p_5=([1: 1: 1],1)$ and $p_6=[1:u:1]$, where $u\in\K^{**}$;
\item[(48.6)] $p_2 = ([1: 0: 0],t)$, $p_3=[0: 1: 0]$, $p_5=([1: 1: 1],1)$ and $p_6=[0:0:1]$, where $t\in\K^{**}$;
\item[(48.7)] $p_2 = ([1: 0: 0],t)$, $p_3=[0: 1: 0]$, $p_5=([1: 1: 1],0)$ and $p_6=[1:0:1]$, where $t\in\K^{**}$;
\item[(48.8)] $p_2 = ([1: 0: 0],\infty)$, $p_3=[0: 1: 0]$, $p_5=([1: 1: 1],\infty)$ and $p_6=[1:-1:-1]$;
\item[(49.1)] $p_3 = ([0: 1: 0],\infty,1)$, $p_4=[1: 0: 0]$, $p_5=[1: 1: 1]$ and $p_6=[1:t:u]$, where $u\in\K^{**}$, $t\in\K$, $t\ne u,u^2$;
\item[(49.2)] $p_3 = ([0: 1: 0],\infty,1)$, $p_4=[1: 0: 0]$, $p_5=[1: 1: 1]$ and $p_6=[1:t:1]$, where $t\in\K$, $t\ne 1$;
\item[(49.3)] $p_3 = ([0: 1: 0],\infty,1)$, $p_4=[1: 0: 0]$, $p_5=[1: 1: 1]$ and $p_6=[1:t:t]$, where $t\in\K^{**}$;
\item[(50.1)] $p_1 = [1: 0: 0]$, $p_2=[0: 1: 0]$, $p_3=[0: 0: 1]$, $p_5=([1:1:1],t)$ and $p_6=[u:v:1]$, where $t,u,v\in\K^{**}$, $v\ne u$, $t\ne u(v - 1)/(v(u - 1))$;
\item[(50.2)] $p_1 = [1: 0: 0]$, $p_2=[0: 1: 0]$, $p_3=[0: 0: 1]$, $p_5=([1:1:1],0)$ and $p_6=[u:v:1]$, where $u,v\in\K^{**}$, $v\ne u$;
\item[(50.3)] $p_1 = [1: 0: 0]$, $p_2=[0: 1: 0]$, $p_3=[0: 0: 1]$, $p_5=([1:1:1],t)$ and $p_6=[u:1:1]$, where $t,u\in\K^{**}$;
\item[(50.4)] $p_1 = [1: 0: 0]$, $p_2=[0: 1: 0]$, $p_3=[0: 0: 1]$, $p_5=([1:1:1],(1-v)/(1-u))$ and $p_6=[u:v:1]$, where $u,v\in\K^{**}$, $v\ne u$;
\item[(50.5)] $p_1 = [1: 0: 0]$, $p_2=[0: 1: 0]$, $p_3=[0: 0: 1]$, $p_5=([1:1:1],0)$ and $p_6=[1:v:1]$, where $v\in\K^{**}$;
\item[(51.1)] $p_1 = [1: 0: 0]$, $p_2=[0: 1: 0]$, $p_4=([0: 0: 1],t)$, $p_5=[1:1:1]$ and $p_6=[u:v:1]$, where $t,u,v\in\K^{**}$, $v\ne u$, $t\ne v(u - 1)/(u(v - 1))$;
\item[(51.2)] $p_1 = [1: 0: 0]$, $p_2=[0: 1: 0]$, $p_4=([0: 0: 1],t)$, $p_5=[1:1:1]$ and $p_6=[u:1:1]$, where $t,u\in\K^{**}$;
\item[(51.3)] $p_1 = [1: 0: 0]$, $p_2=[0: 1: 0]$, $p_4=([0: 0: 1],1)$, $p_5=[1:1:1]$ and $p_6=[u:v:1]$, where $u,v\in\K^{**}$, $v\ne u$;
\item[(51.4)] $p_1 = [1: 0: 0]$, $p_2=[0: 1: 0]$, $p_4=([0: 0: 1],t)$, $p_5=[1:1:1]$ and $p_6=[u:u:1]$, where $t,u\in\K^{**}$;
\item[(51.5)] $p_1 = [1: 0: 0]$, $p_2=[0: 1: 0]$, $p_4=([0: 0: 1],1)$, $p_5=[1:1:1]$ and $p_6=[u:1:1]$, where $u\in\K^{**}$;
\item[(52.1)] $p_1 = [1: 0: 0]$, $p_3=([0: 1: 0],t)$, $p_4=[0: 0: 1]$, $p_5=[1:1:1]$ and $p_6=[u:v:1]$, where $t,u,v\in\K^{**}$, $v\ne u$, $t\ne  (v-u)/(u(v - 1))$;
\item[(52.2)] $p_1 = [1: 0: 0]$, $p_3=([0: 1: 0],t)$, $p_4=[0: 0: 1]$, $p_5=[1:1:1]$ and $p_6=[u:0:1]$, where $t,u\in\K^{**}$;
\item[(52.3)] $p_1 = [1: 0: 0]$, $p_3=([0: 1: 0],t)$, $p_4=[0: 0: 1]$, $p_5=[1:1:1]$ and $p_6=[0:v:1]$, where $t,v\in\K^{**}$;
\item[(52.4)] $p_1 = [1: 0: 0]$, $p_3=([0: 1: 0],t)$, $p_4=[0: 0: 1]$, $p_5=[1:1:1]$ and $p_6=[u:u:1]$, where $t,u\in\K^{**}$;
\item[(52.5)] $p_1 = [1: 0: 0]$, $p_3=([0: 1: 0],t)$, $p_4=[0: 0: 1]$, $p_5=[1:0:1]$ and $p_6=[0:1:1]$, where $t\in\K^{**}$;
\item[(53.1)] $p_1 = [1: 0: 0]$, $p_2=[0: 1: 0]$, $p_3=[0: 0: 1]$, $p_4=[1:1:1]$, $p_5=[t:u:1]$ and $p_6=[v:w:1]$, where $t,u,v,w\in\K^{**}$, $u\ne t$, $v\ne t$, $w\ne u, v, uv/t, (uv + t - u - v)/(t - 1), uv(t - 1)/(tu + tv - uv - t)$;
\item[(53.2)] $p_1 = [1: 0: 0]$, $p_2=[0: 1: 0]$, $p_3=[0: 0: 1]$, $p_4=[1:1:1]$, $p_5=[t:1:1]$ and $p_6=[v:w:1]$, where $t,v,w\in\K^{**}$, $v\ne t$, $w\ne v, v/t$;
\item[(53.3)] $p_1 = [1: 0: 0]$, $p_2=[0: 1: 0]$, $p_3=[0: 0: 1]$, $p_4=[1:1:1]$, $p_5=[t:u:1]$ and $p_6=[v(t-1): uv + t - u - v:t-1]$, where $t,u,v\in\K^{**}$, $u\ne t$, $v\ne t$;
\item[(53.4)] $p_1 = [1: 0: 0]$, $p_2=[0: 1: 0]$, $p_3=[0: 0: 1]$, $p_4=[1:1:1]$, $p_5=[t:1:1]$ and $p_6=[1:w:1]$, where $t,w\in\K^{**}$, $w\ne 1/t$;
\item[(53.5)] $p_1 = [1: 0: 0]$, $p_2=[0: 1: 0]$, $p_3=[0: 0: 1]$, $p_4=[1:1:1]$, $p_5=[1:w:w]$ and $p_6=[1:w:1]$, where $w\in\K^{**}$.
\end{itemize}
\end{proof}

\begin{remark}
Concerning type 8.1 in Table \ref{table4.3}, we note that
\begin{itemize}
\item the choice of $p_1$ and $p_3$ depends on two parameters, since $p_1$ and $p_3$ are proper points in the plane;
\item the choice of $p_2$, $p_4$, $p_5$ and $p_6$ depends only on one parameter, because they are infinitely near other points and they are not satellite.
\end{itemize}
Since the total number of parameters of the configuration is 8, one expects that there is an automorphism of the plane that maps the configuration of the six points in a fixed one.

First, note that, up to automorphisms of the plane, we may choose
\begin{itemize}
\item $p_1=[0:0:1]$ and $p_2=(p_1,\infty)$, that means that $p_2$ is the infinitely near point of order 1 to $p_1$ in the direction of the line $x=0$;
\item $p_3=[1:0:0]$ and $p_4=(p_3,0)$, that means that $p_4$ is the infinitely near point of order 1 to $p_3$ in the direction of the line $z=0$;
\item $p_5=(p_3,0,-1)$, that is the infinitely near point of order 2 to $p_3$, and infinitely near of order 1 to $p_4$, such that the conic $y^2+xz=0$ passes through $p_5$.
\end{itemize}
In particular, the conic with equation $y^2+xz=0$ is the unique (irreducible) conic passing through $p_1, \ldots, p_5$.
Now the automorphisms of the plane that fix $p_1, p_2, p_3, p_4$ and $p_5$ have the following form:
\begin{equation}\label{auto_b}
[x:y:z] \mapsto [x : b y : b^2 z]
\end{equation}
for some $b\in\K^*$.
Therefore, the point $p_6$ has standard coordinates $(p_3,0,t)$ for some $t\in\K^{*}$, $t\ne-1$, because $p_6$ is not satellite, it must be different from $p_5$ and it cannot be 0, otherwise $p_6$ would be collinear with $p_3$ and $p_4$. 

One checks that the automorphisms \eqref{auto_b} cannot change the point $p_6$ to another one.

We conclude that the configuration of the six points depends on one parameter, up to automorphisms of the plane.

We may say that the configuration 8.1 has not the expected number of parameters.

Similarly, one may check that the same occurs for types 13.1, 16.1, 23.2, 48.7, while all other types have always the expected number of parameters which is the number of parameters of the configuration subtracted by 8, the dimension of the automorphisms of the plane.
\end{remark}

\section{Height of points with respect to a plane Cremona map}
\label{sec:height}

\begin{definition}
Let $\varphi$ be a plane Cremona map.
Let us define the \emph{height $h_\varphi(p)$ of a point $p\in\BB(\PP^2)$ with respect to $\varphi$} as follows:
\begin{equation*}\label{eq:height}
\h_\varphi(p)=\begin{cases}
0 & \text{if $p$ is not a base point of $\varphi$,} \\
1 & \text{if $p$ is a proper base point of $\varphi$,} \\
n+1 & \text{if $p$ is a base point of $\varphi$ and $p\infnear_n p'\in\PP^2$.}
\end{cases}
\end{equation*}
\end{definition}

%
%

\begin{notation}\label{not:psi(p)}
Let $\varrho$ be an involutory ordinary quadratic map and let $p_1,p_2,p_3\in\PP^2$ be its base points.
Denote by $\ell_1$ ($\ell_2$, $\ell_3$, resp.) the line passing through $p_2$ and $p_3$ ($p_1$ and $p_3$, $p_1$ and $p_2$, resp.) and denote by $T$ the triangle $\ell_1\cup\ell_2\cup\ell_3$.

Let us define a bijection $\bar\varrho\colon\BB(\PP^2)\to\BB(\PP^2)$ induced by $\varrho$ as follows:
\begin{itemize}
\item $\bar\varrho(p) = p$, if $p=p_i$, $i=1,2,3$;
\item $\bar\varrho(p) = \varrho(p)$, if $p\in\PP^2\setminus T$;
\item $\bar\varrho(p)$ is the point infinitely near $p_i$ of order 1 in the direction of the strict transform of the line passing through $p_i$ and $p$, if $p\in\ell_i\setminus\{p_j,p_k\}$, $\{i,j,k\}=\{1,2,3\}$;
\item $\bar\varrho(p)$ is the point infinitely near $p_j$ of order 1 in the direction of the line $\ell_i$, if $p$ is the point infinitely near $p_i$ of order 1 in the direction of the line $\ell_j$, where $\{i,j\}\subset\{1,2,3\}$;
\item $\bar\varrho(p)$ is the point $q\in \ell_i$ such that the line passing through $p_i$ and $q$ is the strict transform of the line passing through $p_i$ in the direction of the point $p$, if $p$ is infinitely near $p_i$ of order 1 (not lying on $\ell_j$ and $\ell_k$, $\{i,j,k\}=\{1,2,3\}$);
\item $\bar\varrho(p)$ is the point infinitely near $\varrho(p')$ of order $n$ in the direction of the strict transform of a plane curve $C$, if $p$ is infinitely near $p'\in\PP^2$ and $C$ is a curve passing through $p$.
\end{itemize}
Let us say that $\bar\varrho(p)\in\BB(\PP^2)$ is the point corresponding to $p\in\BB(\PP^2)$ via $\varrho$.
\end{notation}

\begin{proposition}\label{PropMAC118a}
Let $p_1, p_2, p_3$ be the base points of an involutory ordinary quadratic plane Cremona map $\varrho\colon \mathbb{P}^2 \dashrightarrow \mathbb{P}^2$.
Let $\varphi\colon \mathbb{P}^2 \dashrightarrow \mathbb{P}^2$ be a plane Cremona map of degree $d>1$ with base points $p_4,\ldots,p_r$ and possibly $p_1,p_2,p_3$. Denote by $m_i$ the multiplicity of $\varphi$ at $p_i$, $i=1,\ldots,r$
(that is $m_i=0$ if $p_i$ is not a base point of $\varphi$, $i=1,2,3$).
Denote by $\bar\varrho(p)$ the (possibly infinitely near) point corresponding to $p$ via $\varrho$ as in Notation \ref{not:psi(p)}.

Then, the composite map $\varphi\circ \varrho^{-1}=\varphi\circ \varrho$ has degree $d-\varepsilon$, where
\[
\varepsilon = m_1+m_2+m_3-d,
\]
and it has $\bar\varrho(p_i)$, $i=4,\ldots,r$, as base point of multiplicity $m_i$. Furthermore, it has multiplicity $m_i-\varepsilon\geq0$ at $p_i$, $i=1,2,3$ (that is, $p_i$ is not a base point of $\varphi\circ \varrho$ when $\varepsilon=m_i$).
\end{proposition}

\begin{proof}
Cf.\ Proposition 4.2.5 in \cite{MAC}.
\end{proof}

Next Lemma and Proposition have been proved in \cite{CN}.

\begin{lemma}\label{lem:oq+-1}
Let $\varphi$ be a plane Cremona map and $\varrho$ be an involutory ordinary quadratic map.
If $p\in\BB(\PP^2)$ and $\bar p=\bar\varrho(p)\in\BB(\PP^2)$ as in Notation \ref{not:psi(p)}, then
\[
-1 \leq \h_\varphi(p)- \h_{\varphi\circ\varrho}(\bar p)\leq 1.
\]
\end{lemma}

\begin{proposition}\label{oq>height}
Let $\varphi$ be a plane Cremona map. Then
\[
\oq(\varphi) \geq \max\{ \h_\varphi(p) \mid p \in \BB(\PP^2) \}.
\]
\end{proposition}

\section{Ordinary quadratic length of some quartic maps}\label{sec:oql}

In next Table \ref{table3} we show the decomposition of some plane Cremona map of degree four in ordinary quadratic maps.

\begin{longtable}[c]{|c|c|}
\hline
\rule{0ex}{2.5ex}$\sharp$ & A decomposition $\alpha_n\circ\textcolor{red}{\sigma}\circ\alpha_{n-1}\circ\textcolor{red}{\sigma}\circ\cdots\circ\alpha_1\circ\textcolor{red}{\sigma}\circ\alpha_0$ \\  \hline

   &  $\alpha_7=[z: x + 2z: y + z], \alpha_6 = [y: -x + z: x], \alpha_5=[x: y - z: z]$ \\
1.1    & $\alpha_4 =[-x + z: x: y + z],\alpha_3=[y: -x - y: x + y + z], \alpha_2=[x: -y + z: y],$ \\
    &  $\alpha_1=[-x: x + y: z], \alpha_0  = [-y + z: y: x]$ \\
\hline
   & $\alpha_6=[x + z: -z: y + z], \alpha_5=[x: y - z: z]$, \\
2.1  &$\alpha_4=[-x - 2y - 4z: x + 2y + 3z: y + z], \alpha_3=[y: -x + z: x],$ \\
  & $\alpha_2=[x: -x - y + z: -z], \alpha_1=[-x: y: x + z], \alpha_0=[-y: -x - y + z: y - z]$ \\
\hline
& $\alpha_7=[2y + z: -2y - 3z: 2x + 2y + 7z], \alpha_6=[-y: -z - 2x: x]$, \\
3.1  &$ \alpha_5=[-2y: -7y + z - 4x: x + 2y],\alpha_4=[-2y - 2z + 2x: 2y + z: z],$ \\
  & $\alpha_3=[y + 2x: -2x - y - z: x],\alpha_2=[z - x - y: x: y]$, \\
  & $\alpha_1=[-x: x + y: z], \alpha_0=[-y + z: y: x]$ \\
\hline
   & $\alpha_6=[-x + y + z: z: x + z], \alpha_5=[-x: y: x + z]$, \\
4.1  &$\alpha_4=[-x: -y + x + z: y], \alpha_3=[y: x + y: z], \alpha_2=[x + y - z: -x + z: x],$ \\
  & $ \alpha_1=[z: y: x + y], \alpha_0=[-y + z: y: x]$ \\
\hline
  & $\alpha_4=[2x + z: z: 2y + z], \alpha_3=[-y: -2z - x: z], \alpha_2=[x - z: x + z: 2y],$ \\[-0.75ex]
\raisebox{1ex}{5.1}  & $\alpha_1=[z: y: x + y], \alpha_0=[-y + z: y: x]$	\\
\hline
   & $\alpha_6=[2y + z: -3z: 9x + y + 11z], \alpha_5=[-3y: -3z - 2x: x]$, \\
6.1  &$\alpha_4=[z: 2x - y + 4z: -x - 2z], \alpha_3=[y + 2x: z: -x],$ \\
  & $\alpha_2=[2y: x - z: x + z], \alpha_1=[z:y: x + y], \alpha_0=[-y + z: y: x]$ \\
\hline
  & $\alpha_4=[-x + y: 4y + 4z: x + y], \alpha_3=[y: -z: x + z],$ \\[-0.75ex]
\raisebox{1ex}{7.1}  & $\alpha_2=[-x + y + z: x - y + z: x + y - z],\alpha_1=[z: y: x + y], \alpha_0=[-y + z: y: x]$	\\
\hline
  & $\alpha_3=[t \left(t y +z \right): -x \left(t +1\right)^{2}: t \left(y -z \right)], \alpha_2=[y \left(t +1\right): -t y +t z +x: x +y -z],$ \\[-0.75ex]
\raisebox{1ex}{8.1}  & $\alpha_1=[z: y: x + y], \alpha_0=[-y + z: y: x]$	\\
\hline
  & $\alpha_4=[y - x: y + z: x - y - z], \alpha_3=[y: -z: x + z],$ \\[-0.75ex]
\raisebox{1ex}{9.1}  & $\alpha_2=[y: -y + z: x + y - z],\alpha_1=[z: y: x + y], \alpha_0=[-y + z: y: x]$	\\
\hline
   & $\alpha_5=[-2x + z: -x: -x - y], \alpha_4=[-z: y + z: x]$, \\
10.1  &$\alpha_3=[x + y: y: z], \alpha_2=[-z: y - x: x + z],$ \\
  & $\alpha_1=[x + z: y + z + x: -x], \alpha_0=[-y + z: x - y: y]$ \\
\hline
   & $\alpha_5=[x + y: -y: y - z], \alpha_4=[-x: -y - z: z]$, \\
12.1  &$\alpha_3=[-y + z: -z: -x], \alpha_2=[-x - y: -y: x + 2y - z],$ \\
  & $\alpha_1=[x: y: y + z], \alpha_0=[-x + y: -y: y - z]$ \\
\hline
   &  $\alpha_7=[z: -y - z: -z - x], \alpha_6 =[-x: -y - z: z], \alpha_5=[-x + z + 3y: x - 2y: y]$ \\
16.2    & $\alpha_4 =[-z: x - z: -x - y],\alpha_3=[y: -x - y + z: x + y], \alpha_2=[y: x - z: z],$ \\
    &  $\alpha_1=[z: y: x + y], \alpha_0  =[-y + z: y: x] $ \\
\hline
 & $\alpha_4=[t^{2} x +z \,t^{2}-2 t x -z t -z u +x +z: 2 t^{2} x -y \,t^{2}+2 z \,t^{2}-4 t x+$\\
   & $+2 y t -4 z t -2 z u +2 x -y +z: t^{2} x +z \,t^{2}-2 t x -z t -z u +x]$, \\
37.1  &$\alpha_3=[-x: -z \,t^{2}+2 z t -y -z: z],$\\
  & $\alpha_2=[z \,t^{3}-t^{2} x +y \,t^{2}-z \,t^{2}-z t u +2 t x -y t -y u -x: -z t -y: z \left(t -1\right)^{3}],$ \\
  & $\alpha_1=[\left(t -1\right) x: t x +y t -y +z: -x -z], \alpha_0=[-x: x - y: x - z]$ \\
\hline
  & $\alpha_4=[uy + z: 2uy + x + y + 2z: uy + y + z], \alpha_3=[y - z: z: x],$ \\[-0.75ex]
\raisebox{1ex}{37.2}  & $\alpha_2=[ux - y + z: -x: y],\alpha_1=[x - y: -x - z: -z], \alpha_0=[-x + z: -x + y: x]$ \\
\hline
\addlinespace 
\caption{Decompositions into ordinary quadratic maps of some quartic planar non-de Jonqui\`eres maps.} \label{table3}
\end{longtable}

\begin{lemma}\label{oq sharp_37.n}
Let $\varphi_{37.n}$ be the map $37.n$ in Table \ref{table4.4} where $n=1,2$. Then, $\oq(\varphi_{37.n}) =4.$
\end{lemma}

\begin{proof}
The enriched weighted proximity graph of $\varphi_{37.n}$ is listed in Table \ref{table4.3}, where $n=1,2$. Note that if $n=2$ then $p_1, p_4, p_5$ are collinear: call $\ell$ such a line.
Since the height of $p_6$ with respect to $\varphi_{37.n}$ is $h_{\varphi_{37.n}}(p_6)=3$, one has $\oq(\varphi_{37.n})\geqslant 3$ by Proposition \ref{oq>height}.
On the other hand, the decomposition of $\varphi_{37.n}$ in Table \ref{table3} implies that $\oq(\varphi_{37.n}) \leqslant 4$.

Suppose by contradiction that $\oq(\varphi_{37.n}) =3$. Then, there should exist an ordinary quadratic map $\rho$ such that $\oq(\varphi_{37.n}\circ\rho^{-1}) = 2$.
We see that $\rho$ must be based at $p_4$, and additionally if $n=2$ then $\rho$ must not base at another point lying on the line $\ell$, otherwise the maximum height of the base points with respect to $\varphi_{37.n}\circ\rho^{-1}$ would be still 3, a contradiction with Proposition \ref{oq>height}.

If the other two base points of $\rho$ are $p_i$ and $p_j$ where $i\neq j$, and $i,j \in \{1,2,3\}$ if $n=1$ or $i,j \in \{2,3\}$ if $n=2$, then $\varphi_{37.n}\circ\rho^{-1}$ is a cubic map equivalent to type 24 in our classification in \cite{CN}, that we showed to have ordinary quadratic length 3, a contradiction.

Otherwise, one sees that $\varphi_{37.n}\circ\rho^{-1}$ has degree $\geqslant5$, and therefore it cannot have ordinary quadratic length 2, another contradiction.

We conclude that $\oq(\varphi_{37.n})=4$ where $n=1,2$. 
\end{proof}

The following result will be useful to compute ordinary quadratic length of plane Cremona maps.

\begin{lemma}
Let $\varphi$ be a plane Cremona map with $p_1, \ldots, p_r$ being the base points of multiplicities $m_1, \ldots,m_r$ respectively, such that $p_1$ is a proper point, $p_{i+1}$ is infinitely near $p_i$ of order 1 for any $i=1,\ldots,r-1$, i.e.,
\begin{center}
\scalebox{0.5}{%
\begin{tikzpicture}[->,>=stealth',shorten >=2pt,auto,node distance=2.5cm,ultra thick,baseline=-1.25ex]
\tikzstyle{every state}=[fill=white,draw=black,text=black]
\node[state,draw=red,text=red,label=below: $p_1$] (A)              {\huge $m_1$};
\node[state,label=below: $p_2$] (B) [right of=A] {\huge $m_2$};
\node[state,label=below: $p_3$] (C) [right of=B] {\huge $m_3$};
\node[state, draw=none] (D) [right of=C] {\huge  $\ldots$};
\node[state,label=below: $p_r$] (E) [right of=D] {\huge  $m_r$};
\path (D) edge (C);
\path (C) edge (B);
\path (B) edge (A);
\path (E) edge (D);
\end{tikzpicture}%
}.
\end{center}
Suppose that $\oq(\varphi) = r$. Then, $m_r = 1$ and $\deg(\varphi) \leq m_1+\cdots+m_r+1$.
\end{lemma}
\begin{proof}
Since $\oq(\varphi) = r$, there exist $\rho_1, \ldots, \rho_r$ are ordinary quardratic maps such that $\varphi = \rho_r \circ \ldots \circ \rho_1$. One can see that $\rho_1$ must have base point $p_1$ and $\rho_i$ must have base point $\bar{p_i}$ where $\bar{p_i}$ is the image of $p_i$ via $\varphi \circ \rho_1^{-1} \circ \ldots \circ \rho_{i-1}^{-1}$ (for any $i=2, \ldots, r$), otherwise the maximum height of the base points of the map $\varphi \circ \rho_1^{-1} \circ \ldots \circ \rho_i^{-1} = \rho_r \circ \ldots \circ \rho_{i+1}$ is $r-i+1$, meanwhile $\oq(\varphi \circ \rho_1^{-1} \circ \ldots \circ \rho_i^{-1}) = r-i$ (for any $i=1,\ldots,r$), which contradicts Proposition \ref{oq>height}. Moreover, by Notation \ref{not:psi(p)}, one has $\mult_{\varphi \circ \rho_1^{-1} \circ \ldots \circ \rho_{i-1}^{-1}}(\bar{p_i}) = m_i$ for any $i=2, \ldots,r$. In particular, $\bar{p_r}$ is a base point of $\varphi \circ \rho_1^{-1} \circ \ldots \circ \rho_{r-1}^{-1} = \rho_r$, which is an ordinary quadratic map, of multiplicity $m_r$, so that $m_r=1$.\\

Now $\rho_1$ has base point $p_1$ and two other base points $q_1,t_1$ where $\mult_{\varphi }(q_1) = n_1$ and  $\mult_{\varphi}(t_1) = k_1$. By Proposition \ref{PropMAC118a}, the map $\varphi \circ \rho_1^{-1} =  \rho_r \circ \ldots \circ \rho_2$ has degree $d_1 = d- \varepsilon_1$, where $d = \deg(\varphi)$ and $\varepsilon_1 = m_1+n_1+k_1 -d$. Since $d\geq n_1+k_1$, we get $\varepsilon_1 \leq m_1$. Therefore, $$d=d_1+\varepsilon_1 \leq d_1 + m_1.$$

Similarly, for each $i=2,\ldots,r$, $\rho_i$ has base point $\bar{p_i}$ and two other base points $q_i,t_i$ where $\mult_{\varphi \circ \rho_1^{-1} \circ \ldots \circ \rho_{i-1}^{-1}}(q_i) = n_i$ and $\mult_{\varphi \circ \rho_1^{-1} \circ \ldots \circ \rho_{i-1}^{-1}}(t_i) = k_i$.  By Proposition \ref{PropMAC118a}, the map $\varphi \circ \rho_1^{-1}\circ \ldots \circ \rho_i^{-1} =  \rho_r \circ \ldots \circ \rho_{i+1}$ has degree $d_i = d_{i-1}- \varepsilon_i$, where $d_{i-1} = \deg(\varphi \circ \rho_1^{-1}\circ \ldots \circ \rho_{i-1}^{-1})$ and $\varepsilon_i = m_i+n_i+k_i -d_{i-1}$. Since $d_{i-1}\geq n_i+k_i$, we get $\varepsilon_i \leq m_i$. Therefore, $$d_{i-1}=d_i+\varepsilon_i \leq d_i + m_i.$$

Hence, we obtain
$$d\leq d_1 + m_1 \leq (d_2 + m_2) + m_1 \leq \cdots \leq (d_r + m_r) + \cdots +m_1 = 1+m_r + \cdots +m_1.$$
\end{proof}

\begin{remark}
The proof of the previous lemma shows that one could also assume that there are other proximity relations among the points $p_1, \ldots, p_r$.
\end{remark}

The following Lemmas are special cases of the previous one.

\begin{lemma}\label{bounded_degree}
\begin{itemize}
\item[$(i)$] Let $\varphi$ be a plane Cremona map with $p_1, p_2$ being the simple base points, such that $p_1$ is proper point and $p_2$ is infinitely near $p_1$ of order 1, i.e.,
\begin{center}
\scalebox{0.5}{%
\begin{tikzpicture}[->,>=stealth',shorten >=2pt,auto,node distance=2.5cm,ultra thick,baseline=-1.25ex]
\tikzstyle{every state}=[fill=white,draw=black,text=black]
\node[state,draw=red,text=red,label=below: $p_1$] (A)              {\huge 1};
\node[state,label=below: $p_2$] (B) [right of=A] {\huge 1};
\path (B) edge (A);
\end{tikzpicture}%
}.
\end{center}
Suppose that $\oq(\varphi) = 2$. Then, $\deg(\varphi) \leq 3$. 
\item[$(ii)$] Let $\varphi$ be a plane Cremona map with $p_1, p_2,p_3$ being the simple base points, such that $p_1$ is proper point and $p_{i+1}$ is infinitely near $p_i$ of order 1 for any $i=1,2$, i.e., 
\begin{center}
\scalebox{0.5}{%
\begin{tikzpicture}[->,>=stealth',shorten >=2pt,auto,node distance=2.5cm,ultra thick,baseline=-1.25ex]
\tikzstyle{every state}=[fill=white,draw=black,text=black]
\node[state,draw=red,text=red,label=below: $p_1$] (A)              {\huge 1};
\node[state,label=below: $p_2$] (B) [right of=A] {\huge 1};
\node[state,label=below: $p_3$] (C) [right of=B] {\huge 1};
\path (C) edge (B);
\path (B) edge (A);
\end{tikzpicture}%
}.
\end{center}
Suppose that $\oq(\varphi) = 3$. Then, $\deg(\varphi) \leq 4$.
\item[$(iii)$] Let $\varphi$ be a plane Cremona map with $p_1, p_2,p_3,p_4$ being the base points, where $p_1$ is proper point of multiplicity 2, all others are simple base points, $p_{i+1}$ is infinitely near $p_i$ of order 1 for any $i=1,2,3$ and $p_3$ may or may not proximate to $p_1$, i.e., 
\begin{center}
\scalebox{0.5}{%
\begin{tikzpicture}[->,>=stealth',shorten >=2pt,auto,node distance=2.5cm,ultra thick,baseline=-1.25ex]
\tikzstyle{every state}=[fill=white,draw=black,text=black]
\node[state,draw=red,text=red,label=below: $p_1$] (A)              {\huge 2};
\node[state,label=below: $p_2$] (B) [right of=A] {\huge 1};
\node[state,label=below: $p_3$] (C) [right of=B] {\huge 1};
\node[state,label=below: $p_4$] (D) [right of=C] {\huge 1};
\path (D) edge (C);
\path (C) edge (B);
\path (B) edge (A);
\path (C) edge[bend right,dashed] (A);
\end{tikzpicture}%
}.
\end{center}
Suppose that $\oq(\varphi) = 4$. Then, $\deg(\varphi) \leq 6$.
\end{itemize}
\end{lemma}

\begin{lemma}\label{oq sharp_m.1}
Let $\varphi_{m.1}$ be the map m.1 in Table \ref{table4.4} where $m=10,12$. Then, $\oq(\varphi_{m.1})=5.$
\end{lemma}

\begin{proof}
The enriched weighted proximity graph of $\varphi_{m.1}$ is listed in Table \ref{table4.3}, where $m=10,12$. Note that the height of $p_6$ with respect to $\varphi_{m.1}$ is 4, hence $\oq(\varphi_{m.1})\geqslant4$ by Proposition \ref{oq>height}.
On the other hand, the decomposition of $\varphi_{m.1}$ in Table \ref{table3} implies that $\oq(\varphi_{m.1})\leqslant5$.

Suppose by contradiction that $\oq(\varphi_{m.1}) =4$. Then, there should exist an ordinary quadratic map $\rho$ such that $\oq(\varphi_{m.1}\circ\rho^{-1}) = 3$. In particular, $\rho$ must be based at $p_3$ and not at a point lying on the line passing through $p_3$ and $p_4$, otherwise the maximum height of the base points with respect to $\varphi_{m.1}\circ\rho^{-1}$ would be still 4, a contradiction with Proposition \ref{oq>height}. There are two cases, either $\rho$ is based at $p_1$ or $\rho$ is not based at $p_1$.

In the former case, the map $\varphi_{m.1}\circ\rho^{-1}$ has degree 4 and its enriched weighted proximity graph is type 37.n  in our table, where $n=2$ if $m=10$ and $n=1$ if $m=12$. By Lemma \ref{oq sharp_37.n}, we have $\oq(\varphi_{m.1}\circ\rho^{-1}) = 4$, a contradiction with the hypothesis.

In the latter case, the map $\varphi_{m.1}\circ\rho^{-1}$ has degree 6 and its weighted proximity graph is:
\begin{center}
\scalebox{0.5}{%
\begin{tikzpicture}[->,>=stealth',shorten >=2pt,auto,node distance=2.5cm,ultra thick,baseline=-1.25ex]
  \tikzstyle{every state}=[fill=white,draw=black,text=black]

  \node[state,draw=red,text=red] (A)              {\huge 4};
  \node[state,draw=red,text=red] (B) [right of=A] {\huge 2};
  \node[state] (C) [right of=B] {\huge 2};
  \node[state,draw=red,text=red] (D) [right of=C] {\huge 2};
  \node[state,draw=red,text=red] (E) [right of=D] {\huge 2};
  \node[state,draw=red,text=red] (F) [right of=E] {\huge 1};
  \node[state] (G) [right of=F] {\huge 1};
  \node[state] (H) [right of=G] {\huge 1};
\path (H) edge (G);
\path (G) edge (F);
\path (C) edge (B);
\end{tikzpicture}%
}
\end{center}
a contradiction with Lemma \ref{bounded_degree} $(ii)$.

Hence, we conclude that $\oq(\varphi_{m.1})=5$ where $m=10,12$.
\end{proof}

\begin{lemma}\label{oq sharp_k.1}
Let $\varphi_{k.1}$ be the map $k.1$ in Table \ref{table4.4} where $k=1,3$. Then, $\oq(\varphi_{k.1}) =7.$
\end{lemma}

\begin{proof}
The enriched weighted proximity graph of $\varphi_{k.1}$ is listed in Table \ref{table4.3}, where $k=1,3$. Since the height of $p_6$ with respect to $\varphi_{k.1}$ is 6, one has $\oq(\varphi_{k.1})\geqslant6$ by Proposition  \ref{oq>height}.
On the other hand, the decomposition of $\varphi_{k.1}$ in Table \ref{table3} implies that $\oq(\varphi_{k.1}) \leqslant 7$.

Suppose by contradiction that $\oq(\varphi_{k.1}) =6$. Then, there should exist an ordinary quadratic map $\rho_1$ such that $\oq(\varphi_{k.1}\circ\rho^{-1}_1) = 5$. In particular, $\rho_1$ must be based at $p_1$ and not at a point lying on the line passing through $p_1$ and $p_2$, otherwise the maximum height of the base points with respect to $\varphi_{k.1}\circ\rho^{-1}_1$ would be still 6, a contradiction with Proposition \ref{oq>height}. So the map $\varphi_{k.1}\circ\rho_1^{-1}$ has degree 6 and its enriched weighted proximity graph is:
\begin{center}
\scalebox{0.5}{%
\begin{tikzpicture}[->,>=stealth',shorten >=2pt,auto,node distance=2.5cm,ultra thick,baseline=-1.25ex]
  \tikzstyle{every state}=[fill=white,draw=black,text=black]

  \node[state,draw=red,text=red,label=below: $p'_0$] (A)              {\huge 4};
  \node[state,draw=red,text=red,label=below: $p'_1$] (B) [right of=A] {\huge 2};
  \node[state,draw=red,text=red,label=below: $p'_2$] (C) [right of=B] {\huge 2};
  \node[state,draw=red,text=red,label=below: $p'_3$] (D) [right of=C] {\huge 2};
  \node[state,label=below: $p'_4$] (E) [right of=D] {\huge 2};
  \node[state,label=below: $p'_5$] (F) [right of=E] {\huge 1};
  \node[state,label=below: $p'_6$] (G) [right of=F] {\huge 1};
  \node[state,label=below: $p'_7$] (H) [right of=G] {\huge 1};
\path (G) edge[bend right,dashed] (E);
\path (H) edge (G);
\path (G) edge (F);
\path (F) edge (E);
\path (E) edge (D);
\end{tikzpicture}%
}
\end{center}
where $p'_1, p'_2, p'_3$ are aligned, and $p'_6$ proximates $p'_4$ only if $k=1$. Then, there should exist an ordinary quadratic map $\rho_2$ such that $\oq(\varphi_{k.1}\circ\rho_1^{-1}\circ\rho_2^{-1}) =4$. In particular, the map $\rho_2$ must be based at $p'_3$ and not at a point lying on the line passing through $p'_3$ and $p'_4$, otherwise the maximum height of the base points of the map $\varphi_{k.1}\circ\rho_1^{-1}\circ\rho_2^{-1}$ is 5, a contradiction with Proposition \ref{oq>height}.

There are now two cases: either $\rho_2$ is based at $p'_0$ or $\rho_2$ is not based at $p'_0$.
In the latter case, 
the map $\varphi_{k.1}\circ\rho_1^{-1}\circ\rho_2^{-1}$ has degree $\geqslant8$ and its weighted proximity graph contains
\[
\scalebox{0.5}{%
\begin{tikzpicture}[->,>=stealth',shorten >=2pt,auto,node distance=2.5cm,ultra thick,baseline=-1.25ex]
  \tikzstyle{every state}=[fill=white,draw=black,text=black]

  \node[state,draw=red,text=red,label=below: $p''_0$] (F) {\huge 2};
  \node[state,label=below: $p''_1$] (G) [right of=F] {\huge 1};
  \node[state,label=below: $p''_2$] (H) [right of=G] {\huge 1};
  \node[state,label=below: $p''_3$] (I) [right of=H] {\huge 1};
\path (H) edge (G);
\path (I) edge (H);
\path (G) edge (F);
\path (H) edge [bend right,dashed] (F);
\end{tikzpicture}%
}
\]
where $p''_2$ proximates $p''_0$ only if $k=1$, a contradiction with Lemma \ref{bounded_degree} $(iii)$.

In the former case, we have two sub-cases:
\begin{enumerate}[$(a)$]
  \item the third base point of $\rho_2$ is either $p'_1$ or $p'_2$; 
  \item the third base point of $\rho_2$ is neither $p'_1$ nor $p'_2$.
\end{enumerate}

In sub-case $(a)$, the enriched weighted proximity graph of the map $\varphi_{k.1}\circ\rho_1^{-1}\circ\rho_2^{-1}$ is type $m.1$ in Table \ref{table4.3}, where if $k=1$ then $m=10$ and if $k=3$ then $m=12$ and by Lemma \ref{oq sharp_m.1}, $\oq(\varphi_{k.1}\circ\rho_1^{-1}\circ\rho_2^{-1}) =5$, that is a contradiction.

In sub-case $(b)$, the map $\varphi_{k.1}\circ\rho_1^{-1}\circ\rho_2^{-1}$ has degree $6$ and its enriched weighted proximity graph is
\begin{center}
\scalebox{0.5}{%
\begin{tikzpicture}[->,>=stealth',shorten >=2pt,auto,node distance=2.5cm,ultra thick,baseline=-1.25ex]
  \tikzstyle{every state}=[fill=white,draw=black,text=black]

  \node[state,draw=red,text=red,label=below: $p''_0$] (A)              {\huge 4};
  \node[state,draw=red,text=red,label=below: $p''_1$] (B) [right of=A] {\huge 2};
  \node[state,draw=red,text=red,label=below: $p''_2$] (C) [right of=B] {\huge 2};
  \node[state,draw=red,text=red,label=below: $p''_3$] (D) [right of=C] {\huge 2};
  \node[state,draw=red,text=red,label=below: $p''_4$] (E) [right of=D] {\huge 2};
  \node[state,label=below: $p''_5$] (F) [right of=E] {\huge 1};
  \node[state,label=below: $p''_6$] (G) [right of=F] {\huge 1};
  \node[state,label=below: $p''_7$] (H) [right of=G] {\huge 1};
\path (G) edge[bend right,dashed] (E);
\path (H) edge (G);
\path (G) edge (F);
\path (F) edge (E);
\end{tikzpicture}%
}
\end{center}
where $p''_1, p''_2, p''_3$ are aligned and $p''_6$ proximates $p''_4$ only if $k=1$. Then, there should exist an ordinary quadratic map $\rho_3$ such that $\oq(\varphi_{k.1}\circ\rho_1^{-1}\circ\rho_2^{-1}\circ\rho_3^{-1}) = 3$. The map $\rho_3$ must be based at $p''_4$ and not at a point lying on the line passing through $p''_4$ and $p''_5$, otherwise the maximum height of the base points of the map $\varphi_{k.1}\circ\rho_1^{-1}\circ\rho_2^{-1}\circ\rho_3^{-1}$ is 4, a contradiction with Proposition \ref{oq>height}.

If $\rho_3$ is not based at $p''_0$, then the map $\varphi_{k.1}\circ\rho_1^{-1}\circ\rho_2^{-1}\circ\rho_3^{-1}$ has degree $\geqslant6$ and its weighted proximity graph contains
\[
\scalebox{0.5}{%
\begin{tikzpicture}[->,>=stealth',shorten >=2pt,auto,node distance=2.5cm,ultra thick,baseline=-1.25ex]
  \tikzstyle{every state}=[fill=white,draw=black,text=black]

  \node[state,draw=red,text=red] (F) [right of=E] {\huge 1};
  \node[state] (G) [right of=F] {\huge 1};
  \node[state] (H) [right of=G] {\huge 1};
\path (H) edge (G);
\path (G) edge (F);
\end{tikzpicture}%
}
\]
a contradiction with Lemma \ref{bounded_degree} $(ii)$.

It remains the case that $\rho_3$ is based at $p''_4$ and $p''_0$.
If the third base point of $\rho_3$ is not $p''_i$, $i \in \{1,2,3\}$, then the map $\varphi_{k.1}\circ\rho_1^{-1}\circ\rho_2^{-1}\circ\rho_3^{-1}$ has degree 6 and its weighted proximity graph again contains
\[
\scalebox{0.5}{%
\begin{tikzpicture}[->,>=stealth',shorten >=2pt,auto,node distance=2.5cm,ultra thick,baseline=-1.25ex]
  \tikzstyle{every state}=[fill=white,draw=black,text=black]

  \node[state,draw=red,text=red] (F) [right of=E] {\huge 1};
  \node[state] (G) [right of=F] {\huge 1};
  \node[state] (H) [right of=G] {\huge 1};
\path (H) edge (G);
\path (G) edge (F);
\end{tikzpicture}%
}
\]
that is a contradiction with Lemma \ref{bounded_degree} $(ii)$.

Otherwise, if the third base point of $\rho_3$ is $p''_i$, $i \in \{1,2,3\}$, then the map $\varphi_{k.1}\circ\rho_1^{-1}\circ\rho_2^{-1}\circ\rho_3^{-1}$ has degree 4 and its enriched weighted proximity graph is type $37.n$ in our Table \ref{table4.3}, where $n=2$  if $k=1$ and $n=1$ if $k=3$. By Lemma \ref{oq sharp_37.n}, we have $\oq(\varphi_{k.1}\circ\rho_1^{-1}\circ\rho_2^{-1}\circ\rho_3^{-1}) = 4$, a contradiction.

Overall, we conclude that $\oq(\varphi_{k.1}) =7$ where $k=1,3$.
\end{proof}

\begin{remark}
Type $1.1$, Type $3.1$, and its inverse Type $16.2$ are the only three of the 123 types that have the largest ordinary quadratic length, equal to 7. The ordinary quadratic lengths of the remaining types can be found in detail in Table \ref{table4.3} .
\end{remark}

\subsection*{Acknowledgements}
This work was completed while the second author was a postdoctoral researcher at Vietnam Institute for Advanced Studies in Mathematics (VIASM) and Tokyo University of Science (TUS). She sincerely thanks VIASM and TUS for their very kind support and hospitality.
The first author is supported by the Departmental Research Funding [FIRD] from the Department of Mathematics and Computer Science of the University of Ferrara.
The second author is supported by Project for Mathematical Science and Cooperation with Engineering (MaSCE) at TUS.

\end{document}